\documentclass[11pt]{amsart}%,twocolumn
  \usepackage{paralist}
\usepackage[margin=2cm]{geometry}

\usepackage{xfrac}

\RequirePackage[english]{babel}
\RequirePackage[utf8]{inputenc}
\RequirePackage[T1]{fontenc} 
\usepackage[utf8]{inputenc}
\usepackage[T1]{fontenc}

\usepackage{amssymb,latexsym, amsthm, pifont}
\usepackage{hyperref}
\usepackage{cancel, soul}
\usepackage{tikz}
\usepackage{verbatim}

\usepackage{pgfplots}
\pgfplotsset{width=7cm}
\makeatletter
\newcommand{\addresseshere}{%
  \enddoc@text\let\enddoc@text\relax
}
\makeatother

\usepackage{bm}
\usepackage{graphicx}
\usepackage{color}

\theoremstyle{plain}
\newtheorem{lemma}{Lemma}[section]

\newtheorem{theorem}[lemma]{Theorem}
\newtheorem{corollary}[lemma]{Corollary}

\theoremstyle{definition}
\newtheorem{definition}[lemma]{Definition}

\newtheorem{remark}[lemma]{Remark}
\newtheorem{question}[lemma]{Question}

\numberwithin{equation}{section}

\def\XXint#1#2#3{{\setbox0=\hbox{$#1{#2#3}{\int}$ }
\vcenter{\hbox{$#2#3$ }}\kern-.5\wd0}}

\newcommand{\R}{\mathbb{R}}
\newcommand{\N}{\mathbb{N}}

\newcommand{\Zz}{\mathbf{Z}}

\newcommand{\mcK}{\mathcal{K}}

\newcommand{\D}{\mathcal{D}}

\newcommand{\loc}{\text{\rm loc}}

\newcommand{\RLF}{\boldsymbol X}

\newcommand{\boxx}{\boldsymbol{x}}
\newcommand{\boy}{\boldsymbol{y}}
\newcommand{\boz}{\boldsymbol{z}}
\newcommand{\bov}{\boldsymbol{v}}
\newcommand{\bobb}{\boldsymbol{b}}

\newcommand{\bog}{\boldsymbol{\gamma}}

\newcommand{\Ig}[1]{I_{\bog^{#1}}}
\newcommand{\Igp}[1]{I_{\bog^{#1}}^{_{\parallel}}}
\newcommand{\Ignp}[1]{I_{\bog^{#1}}^{_{\parallel \complement}}}

\newcommand{\mass}[1]{\mathbb M_{\lambda}\left(#1\right)}

\newcommand{\Ll}{\mathcal L}

\newcommand{\norm}[1]{\lVert#1\rVert}

\newcommand{\Lipl}[2]{\Lip_{#1}#2}
\newcommand{\TTT}{[0,T]}

\DeclareMathOperator{\Uno}{\mathbf{1}}
\DeclareMathOperator{\Lip}{Lip}

\DeclareMathOperator{\clos}{clos}

\DeclareMathOperator{\diver}{div_{\mathit x}}

\DeclareMathOperator{\ddz}{\partial_{s}}
\DeclareMathOperator{\amf}{\mathfrak{a}}
\DeclareMathOperator{\Amf}{\mathfrak{A}}
\DeclareMathOperator{\Ima}{Image}

\DeclareMathAlphabet{\mathcalligra}{T1}{calligra}{m}{n}
\DeclareMathAlphabet{\mathpzc}{OT1}{pzc}{m}{it}

\newcounter{stepnb}
\newcounter{substepnb}

\title[A Directional Lipschitz Extension Lemma and the Continuity Equation]{A Directional Lipschitz Extension Lemma,\\ with Applications to Uniqueness and Lagrangianity  \\ for the Continuity Equation}

\begin{document}

\author{Laura Caravenna}
\address{Laura Caravenna: Dipartimento di Matematica `Tullio Levi Civita', Universit\`a di Padova, via Trieste~63, 35121 Padova, Italy}
\email{laura.caravenna@unipd.it}
\author{Gianluca Crippa}
\address{Gianluca Crippa: Departement Mathematik und Informatik, Universit\"at Basel, Spiegelgasse~1,\break 4051 Basel, Switzerland}
\email{gianluca.crippa@unibas.ch}

\begin{abstract}
We prove a Lipschitz extension lemma in which the extension procedure simultaneously preserves the Lipschitz continuity for two non-equivalent distances. The two distances under consideration are the Euclidean distance and, roughly speaking, the geodesic distance along integral curves of a (possibly multi-valued) flow of a continuous vector field. The Lipschitz constant for the geodesic distance of the extension can be estimated in terms of the Lipschitz constant for the geodesic distance of the original function. This Lipschitz extension lemma allows us to  remove the high integrability assumption on the solution needed for the uniqueness within the DiPerna-Lions theory of continuity equations in the case of vector fields in the Sobolev space~$W^{1,p}$, where $p$ is larger than the space dimension, under the assumption that the so-called ``forward-backward integral curves'' associated to the vector field are trivial for almost every starting point. More precisely, for such vector fields we prove uniqueness and Lagrangianity for weak solutions of the continuity equation that are just locally integrable. Additionally, for such vector fields it is possible to prove almost everywhere uniqueness of (standard) integral curves, which also implies uniqueness of positive measure solutions to the continuity equation with absolutely continuous initial datum.  

\bigskip

\end{abstract}

\maketitle

%:
%:
%:
%\tableofcontents

%:
%:
%:
\section{Introduction}

\subsection{The continuity equation and the ordinary differential equation}
Given a (time-dependent) vector field $\bobb = \bobb (t,\boxx) : \TTT \times \R^n \to \R^n$, we consider 
the Cauchy problem for the continuity equation
\begin{equation}\label{e:pde1}
\left\{
\begin{array}{l}
\partial_t u + \diver (\bobb u) = 0 \\ 
u(t=0,\boxx) = u_0(\boxx) \,,
\end{array}\right.
\end{equation} 
in which $u_0 = u_0(\boxx) : \R^n \to \R$ is a given initial datum and $u = u(t,\boxx) : \TTT \times \R^n \to \R$ is the solution.
The classical theory of characteristics establishes a link between solutions of~\eqref{e:pde1}
and the flow $\RLF = \RLF(t,\boxx) : \TTT \times \R^n \to \R^n$ of the vector field $\bobb$, that is, the solution
of the ordinary differential equation
\begin{equation}\label{e:ode1}
\left\{
\begin{array}{l}
\displaystyle\frac{d}{dt} \RLF(t,\boxx) = \bobb (t,\RLF(t,\boxx)) \\ \\ 
\RLF(0,\boxx) = \boxx \,.
\end{array}\right.
\end{equation} 
Both problems~\eqref{e:pde1} and~\eqref{e:ode1} are classically well posed if the vector field $\bobb$ is regular enough in the space variable, more precisely Lipschitz continuous in space with some uniformity in time, and the unique solution of~\eqref{e:pde1} is given by the formula
\begin{equation}\label{e:sol1}
u(t,\RLF(t,\boxx)) = u_0 (\boxx) J\RLF(t,\boxx)^{-1}\,,
\end{equation}
where $J\RLF(t,\boxx) = {\rm det}\, \nabla \RLF (t,\boxx)$ is the Jacobian of the flow. The formula above expresses the fact that the solution $u$ is Lagrangian, i.e., transported by the flow $\RLF$ and stretched by a factor given by the Jacobian of the flow. 

\medskip

\subsection{The DiPerna-Lions theory}\label{ss:DPL}
The by now classical DiPerna-Lions theory~\cite{DPL} deals with the well-posedness of both problems~\eqref{e:pde1} and~\eqref{e:ode1}, and with the validity of formula~\eqref{e:sol1}, in the case when the
vector field has only Sobolev regularity with respect to the space variable. More specifically, in~\cite{DPL} the authors consider vector fields for which, for a given $1 \leq p \leq \infty$, it holds
\begin{subequations}\label{e:bobb}
\begin{align}
\bobb &\in L^1 \left( \TTT ; W^{1,p}_{\rm loc} (\R^n;\R^n) \right) \,, \label{e:bobb-a} \\
\diver \bobb &\in L^1 \left( \TTT ; L^\infty(\R^n) \right) \,, \label{e:bobb-b} \\ 
\displaystyle \frac{| \bobb (t,\boxx)|}{1 + |\boxx|} &\in L^1 \left( \TTT ; L^1(\R^n) \right) + L^1 \left( \TTT ; L^\infty(\R^n) \right) \,.
\label{e:bobb-c}
\end{align}
\end{subequations}

In this low-regularity context, weak solutions of~\eqref{e:pde1} are defined in the usual distributional sense, by testing the equation with smooth functions and integrating by parts. More precisely, this amounts to requiring the validity of
\begin{equation}\label{e:weakform}
\iint u \left( \partial_t \varphi + \bobb \cdot \nabla \varphi \right) \, dx dt = - \int u_0 \varphi(0,\cdot) \, dx 
\end{equation}
for every $\varphi \in C^\infty_c (\TTT \times \R^n)$. Notice that~\eqref{e:weakform} makes sense provided that $u$, $u_0$, and $\bobb u$ are locally integrable. The suitable notion of solution of~\eqref{e:ode1} is that of regular Lagrangian flow: one consider maps $\RLF = \RLF(t,\boxx) : \TTT \times \R^n \to \R^n$ that solve~\eqref{e:ode1} in distributional sense in time for almost every~$\boxx\in \R^n$, and additionally satisfy the near incompressibility condition
\begin{equation}\label{e:inc}
\Ll^n (B) \leq L \; \Ll^n (\RLF(t,B)) \qquad \text{ for every measurable set $B \subset \R^n$,}
\end{equation}
for a given constant $L>0$. Observe that, for a smooth vector field, condition~\eqref{e:inc} is a consequence of a uniform bound on the divergence. 

Given a vector field $\bobb$ as in~\eqref{e:bobb}, the theory in~\cite{DPL} ensures that:
\begin{itemize}
\item[(1)] There exists a unique regular Lagrangian flow $\RLF$ associated to $\bobb$.
\item[(2)] For every $u_0 \in L^q(\R^n)$, with 
\begin{equation}\label{e:soglia0}
\frac{1}{p} + \frac{1}{q} \leq 1 \,,
\end{equation}
there exists a unique weak solution $u \in L^\infty(\TTT;L^q(\R^n))$, for which a suitable weak version of formula~\eqref{e:sol1} holds, that is, such unique weak solution $u$ is Lagrangian. 
\end{itemize}

Let us focus on the integrability assumption~\eqref{e:soglia0}. 
The need for the integrability assumption $u \in L^\infty(\TTT;L^q(\R^n))$ with $p$ and $q$ as in~\eqref{e:soglia0} comes from the strategy of proof in~\cite{DPL}. The authors derive an energy estimate on the continuity equation~\eqref{e:pde1}, which however needs to be carefully justified in this low-regularity setting. Such justification is achieved via a regularization procedure, in which it is essential to show the convergence to zero of suitable commutators. Roughly speaking, such commutators can be rewritten as integral expressions involving the product of the gradient of the vector field $D\bobb$ with the weak solution $u$, which requires the integrability condition~\eqref{e:soglia0}. Notice that condition~\eqref{e:soglia0} is strictly stronger than the integrability condition $\bobb u \in L^1_{\loc}$ needed for the definition of weak solutions as in~\eqref{e:weakform}. 

\medskip

If the condition~\eqref{e:soglia0} is not satisfied, uniqueness and Lagrangianity can be proved only inside the smaller class of the so-called renormalized solutions, while uniqueness may fail for the larger class of all weak solutions (see~\S~\ref{s:MS} below). In particular, it may happen that weak solutions with insufficient integrability are non-unique. The uniqueness of the regular Lagrangian flow is ensured by the assumptions in~\eqref{e:bobb}, therefore some of these non-unique weak solutions are necessarily non-Lagrangian.

\medskip

Let us remark that the theory in~\cite{DPL} has been extended by Ambrosio~\cite{AMB} to vector fields with bounded variation (the distributional derivative is a measure), that is, replacing the Sobolev assumption in~\eqref{e:bobb} by the condition 
$\bobb \in L^1 \left( \TTT ; BV_{\rm loc} (\R^n;\R^n) \right)$, and more recently by Bianchini and Bonicatto~\cite{BB} removing the assumption of bounded divergence and replacing it by a near incompressibility condition, therefore giving a positive answer to Bressan's compactness conjecture~\cite{Br}. In both cases, uniqueness and Lagrangianity of weak solutions are obtained in the class $u \in L^\infty(\TTT;L^\infty(\R^n))$: the integrability condition~\eqref{e:soglia0} still holds in a weaker sense. 

\subsection{Main results of this paper} 
\label{s:mainres}
As discussed above, weak solutions $u \in L^\infty(\TTT;L^q(\R^n))$ of the continuity equation~\eqref{e:pde1} may be in general non-unique and non-Lagrangian if $p$ and $q$ do not satisfy condition~\eqref{e:soglia0}. Our main result asserts that, for a vector field $\bobb$ 
belonging to $W^{1,p}_{\rm loc} (\R^n;\R^n)$ uniformly in time, where $p>n$, satisfying~\eqref{e:bobb-b} and~\eqref{e:bobb-c}, and such that its associated forward-backward integral curves (as in Definition~\ref{d:fb}) are almost everywhere trivial (in the sense that their image is contained in the image of a (standard) integral curve, see again Definition~\ref{d:fb}),
uniqueness and Lagrangianity hold for weak solutions that are merely locally integrable in time and space.
\begin{theorem}
\label{T:main}
Let $\bobb$ be a vector field satisfying~\eqref{e:bobb-b} and~\eqref{e:bobb-c} and 
\begin{equation}\label{e:bobb-b'}
\bobb \in L^\infty \left( \TTT ; W^{1,p}_{\rm loc} (\R^n;\R^n) \right)
\end{equation}
with $p>n$. Assume moreover that, for almost every initial point, forward-backward integral curves of the vector field~$\bobb$ are trivial. Then, given an initial datum $u_{0} \in L^{1}_{\mathrm{loc}}(\R^{n})$, the Cauchy problem for~\eqref{e:pde1} has a unique weak solution 
$$
u \in L^{1}_{\mathrm{loc}} \left(\TTT\times\R^{n}\right) \,.
$$
Such unique weak solution is Lagrangian and renormalized. 
\end{theorem}
\begin{remark}
The above theorem roughly means that, apart from the assumption on the forward-backward integral curves, the same conclusion as in item~(2) in~\S~\ref{ss:DPL} above holds if $p>n$ and $q=1$, although in such a case one only has 
$$\frac{1}{p} + \frac{1}{q} < 1 + \frac{1}{n}\,,
$$
and not the more demanding condition~\eqref{e:soglia0}.
\end{remark}

\begin{remark} 
In fact, Theorem~\ref{T:main} holds even under the weaker condition~$p>1$ (replacing~$p>n$), provided that we require the vector field to be uniformly continuous in the space variable, uniformly in time (see Definition~\ref{D:cont}).
\end{remark}

\begin{remark}
Theorem~\ref{T:main} can be easily extended to the case where a source term or a linear term of zero order are present in the continuity equation, under suitable integrability conditions on the coefficients. In particular, we can also deal with the transport equation $\partial_t u + \bobb \cdot \nabla u = 0$ instead of the continuity equation~\eqref{e:pde1}.
\end{remark}

\begin{remark}\label{r:b}
It is not clear to us whether for every 
vector field~$\bobb$ satisfying~\eqref{e:bobb-b'},~\eqref{e:bobb-b}, and~\eqref{e:bobb-c}
forward-backward integral curves are almost everywhere trivial. We prove in Corollary~\ref{c:aeuni} that for such a vector field (standard) integral curves are almost everywhere unique. However, we observe in Remark~\ref{r:nouni} that, for a given initial point, uniqueness of standard integral curves does not imply 
the triviality of forward-backward integral curves starting at that point. 
\end{remark}

In fact, the proof of Theorem~\ref{T:main} relies on our previous work~\cite{CC},
in which we presented a conditional proof of the result under the assumption that a certain geometric property holds for the regular Lagrangian flow of~$\bobb$, compare Assumption~3.3 in~\cite{CC}. Such Assumption~3.3 postulates the validity of a suitable directional Lipschitz extension lemma. The proof of such a lemma under the assumptions of Theorem~\ref{T:main} is the main result of the present article, compare Theorem~\ref{T:LipschitzEstention}. We will summarize the conditional proof from~\cite{CC} and briefly describe the setting and the strategy of the proof of Theorem~\ref{T:LipschitzEstention} in~\S~\ref{s:SP} below. 
Before that, let us briefly pause in order to describe some related results.

\medskip

For the well-posedness of problems~\eqref{e:pde1} and~\eqref{e:ode1} the two-dimensional case $n=2$ turns out to be very peculiar, roughly speaking for topological reasons. The analysis in~\cite{ABC} provides a characterization of the autonomous, bounded, divergence-free vector fields $\bobb : \R^2 \to \R^2$ for which uniqueness for the continuity equation~\eqref{e:pde1} holds, in terms of a certain weak Sard property of the stream function associated to the vector field. Such weak Sard property is satisfied, for instance, in the case of Sobolev or $BV$ vector fields, and in turn it implies uniqueness and Lagrangianity of weak solutions $u \in L^{1}_{\mathrm{loc}} \left(\TTT\times\R^{2}\right)$. The proof in~\cite{ABC} is very specific to the autonomous two-dimensional setting and does not appear suitable for extensions to higher dimension, or to the non-autonomous case.

Uniqueness and Lagrangianity of weak solutions with lower integrability than the one required by~\eqref{e:soglia0} can be obtained in some situations carrying a physical meaning, for instance for vanishing viscosity solutions of the two-dimensional Euler equations with vorticity in $L^p$, for any $1 \leq p \leq \infty$, see~\cite{E1,E2} (and also~\cite{E3} for solutions obtained as limit of the vortex-blob approximation scheme).

For the case of one space dimension, the analysis of~\cite{CSC,ABCar} proves that continuous solutions to the nonlinear balance equation $\partial_t u + \partial_x f(u) = g$, where $g$ is bounded, are Lagrangian, under a technical nondegeneracy condition on the flux $f$. Owing to the total ordering of the real line, indeed, using Peano theorem one can suitably define a meaningful notion of pointwise flow of~\eqref{e:ode1}, even though in general one does not obtain the uniqueness of integral curves nor the near incompressibility condition~\eqref{e:inc}. Continuous solutions of balance laws are of interest both for geometric reasons, see for instance~\cite{V,SC} and the references therein, and because the source might act as a control device preserving continuity; this could be the case also for the continuity equation~\eqref{e:pde1}.

\subsection{Examples of non-uniqueness and non-Lagrangianity via convex integration}\label{s:MS}

Let us first of all observe that it is possible to construct vector fields in $W^{1,p}$ for every $1 \leq p < n$ and weak solutions to~\eqref{e:pde1} in $L^q$ for all $1 \leq q < \frac{n}{n-1}$ which are non-unique, non-Lagrangian, and non-renormalized~\cite{B}. The divergence of $\bobb$ belongs however only to $L^p$ and is not bounded. Notice that for such a vector field the theory in~\cite{DPL} guarantees the renormalization property for all weak solutions in $L^q$ with $q > \frac{n}{n-1}$, although the boundedness of the divergence would be required in order to establish their uniqueness. See also~\cite{BEA}, in which the condition that the divergence of the vector field belongs to $L^p$ with $p>n$ plays a key role for the convergence of finite volume schemes for the continuity equation. 

Striking examples of non-uniqueness and non-Lagrangianity with divergence-free vector fields have been constructed very recently by Modena and Sz\'ekelyhidi~\cite{MS} via convex integration. In dimension $n \geq 3$, given integrability indexes $1 \leq p < \infty$ and $1 < q < \infty$ such that
\begin{equation}\label{e:soglia1}
\frac{1}{p} + \frac{1}{q} > 1 + \frac{1}{n-1} \,,
\end{equation}
the authors of~\cite{MS} are able to construct a divergence-free vector field $\bobb \in C \left( \TTT ; W^{1,p} \cap L^{q'} (\R^n) \right)$ such that the continuity equation~\eqref{e:pde1} has infinitely many weak solutions $u \in C \left( \TTT ; L^q(\R^n) \right)$. In the previous expression, $q'$ is such that $\frac{1}{q} + \frac{1}{q'} = 1$, so that weak solutions of~\eqref{e:pde1} are defined in the usual distributional sense. In fact, the  weak solutions constructed in~\cite{MS} also enjoy some additional property, we refer to~\cite{MS} for the precise statements. This result has been extended by the same authors in~\cite{MS2} to the borderline case $q=1$, yielding a continuous and bounded vector field belonging to all spaces $C \left( \TTT ; W^{1,p}(\R^n)\right)$ with $1\leq p < n-1$. In~\cite{MS3} the authors show that the same result holds replacing the condition~\eqref{e:soglia1} by the following one:
\begin{equation}\label{e:soglia2}
\frac{1}{p} + \frac{1}{q} > 1 + \frac{1}{n} \,.
\end{equation}
Let us briefly mention that, under the opposite inequality 
\begin{equation}\label{e:soglia3}
\frac{1}{p} + \frac{1}{q} < 1 + \frac{1}{n} \,,
\end{equation}
the Sobolev conjugate exponent $p^*$ of $p$ is larger than the H\"older conjugate exponent $q'$ of $q$, that is,
$$
\frac{1}{p^*} = \frac{n-p}{np} \stackrel{\eqref{e:soglia3}}{<} 1 - \frac{1}{q} = \frac{1}{q'} \,,
$$
making it is possible to define weak solutions of~\eqref{e:pde1} in the usual distributional sense, as the product $\bobb u$ turns out to be integrable. 

Observe that condition~\eqref{e:soglia3} is strictly weaker than condition~\eqref{e:soglia0}. 
Moreover, as already observed in~\S~\ref{s:mainres}, our Theorem~\ref{T:main} precisely shows 
that in the case $p>n$ (and correspondingly, $q=1$), under the further assumption of
triviality of forward-backward integral curves, condition~\eqref{e:soglia3} is enough to obtain 
uniqueness of weak solutions. The following open questions are therefore extremely natural:

\begin{question}\label{q:b}
For a vector field~$\bobb$ satisfying~\eqref{e:bobb} with $p>n$ are
forward-backward integral curves trivial for almost every starting point? 
That is, is the assumption on the forward-backward integral curves in
Theorem~\ref{T:main} redundant? 
\end{question}

\begin{question}\label{q:q}
Let $p$ and $q$ be integrability indexes satisfying~\eqref{e:soglia3}. Let $\bobb$ be a vector field as in~\eqref{e:bobb}. Does uniqueness of weak solutions to the continuity equation~\eqref{e:pde1} hold in the class $u \in L^\infty \left( \TTT ; L^q(\R^n) \right)$?
\end{question}

We also refer to~\cite{BCD} and~\cite{CL} (both subsequent to the submission of the present paper) for some more recent progress on this topic. In~\cite{BCD} nonuniqueness of {\em positive} weak solutions to the continuity equation~\eqref{e:pde1} is shown under the condition~\eqref{e:soglia2} and this is used in order to show nonuniqueness of a.e.~Lagrangian flows, i.e.~of solutions of~\eqref{e:ode1} for which the near incompressibility condition~\eqref{e:inc} is not required. In~\cite{CL} nonuniqueness of weak solutions to the continuity equation~\eqref{e:pde1} in the class $u \in L^1 \left( \TTT ; L^q(\R^n) \right)$ is shown for $p$ and $q$ integrability indexes not satisfying the DiPerna-Lions integrability assumption~\eqref{e:soglia0}; notice in~\cite{CL} the $L^1$ dependence on time, in contrast to the $L^\infty$ dependence assumed in the above discussion.

\subsection{Strategy of the proof of Theorem~\ref{T:main}}\label{s:SP}

Let us very briefly describe the strategy of proof of~Theorem~\ref{T:main}. For full details we refer to~\cite{CC}. Given a weak solution $u \in L^{1}_{\mathrm{loc}} \left(\TTT\times\R^{n}\right)$  of~\eqref{e:pde1}, it is enough to prove that $u$ is transported by the (unique) regular Lagrangian flow~$\RLF$ associated to~$\bobb$. The natural strategy would be to change variable in the weak formulation~\eqref{e:weakform} using the regular Lagrangian flow, in order to obtain a weak formulation of~\eqref{e:pde1} in Lagrangian variables. This procedure can also be seen as a disintegration of the equation on the integral curves of the vector field, an approach reminiscent of those in~\cite{ABC,BB}. However, due to the lack of Lipschitz regularity of the flow with respect to the space variable (see for instance~\cite{J2,ACM} for examples of Sobolev vector fields with associated regular Lagrangian flow that is not continuous or not Sobolev), we do not obtain yet the distributional formulation of~\eqref{e:ode1}, since we do not obtain the full class of test functions after the change of variable. 

Nevertheless, the Lagrangian theory developed in~\cite{CDL} guarantees some regularity of the regular Lagrangian flow ``on large sets''. This in turn guarantees that the test function we obtain after the change of variable is Lipschitz continuous on a ``large flow-tube'', although with a possibly large Lipschitz constant. We thus need to extend this function to a globally Lipschitz function, and to control the error 
due to the fact that we replace the test function on the (small) set outside the large flow-tube.
The key observation is that, in order to control such an error in the weak formulation in Lagrangian variable,
only the Lipschitz constant of the test function along integral curves is relevant, not the global Lipschitz constant. This is the reason why a directional Lipschitz extension lemma is the key point of our strategy.

\subsection{Directional Lipschitz extension lemma}

In order to handle properly the question of the existence of an extension with the property just mentioned at the end of the last paragraph, it is useful to adopt a more geometric point of view. Let us informally define a function to be directionally Lipschitz continuous if its composition with each integral curve $t \mapsto (t,\RLF(t,\boxx))$ in the regular Lagrangian flow is Lipschitz continuous (see however Definition~\ref{d:multiflow} for the rigorous definition). For the sake of comprehension, let us give an informal version of Theorem~\ref{T:LipschitzEstention}:

\medskip

\noindent {\em Let be given a function defined on a flow-tube, and assume that such a function is Lipschitz continuous and directionally Lipschitz continuous. Then there exists an extension of the function to the entire space which is Lipschitz continuous and directionally Lipschitz continuous. Moreover, the directional Lipschitz constant of the extension can be estimated in terms of the directional Lipschitz constant of the original function.} 

\medskip

From a more abstract point of view, the above result can be seen as a simultaneous Lipschitz extension lemma with respect to two distances, the Euclidean one and the geodesic distance on integral curves of the vector field. Such a result is non-trivial, as these two distances are non-equivalent, and in fact, the geodesic distance is degenerate, since it can assume the value $+\infty$ in case of points belonging to different integral curves. To the best of our knowledge no extension lemmas in similar settings are available in the literature.

\medskip 

The strategy of proof of the directional Lipschitz extension theorem uses an interpolation between the two distances which is customary in the calculus of variations. We approximate the (degenerate) geodesic distance by distances that are equivalent to the Euclidean one, but weight more and more displacements that are transversal to the integral curves (see Definition~\ref{D:distances}). The key lemma is that, for a given Lipschitz continuous function defined on a flow-tube, the Lipschitz constant for the approximated distances converges to the Lipschitz constant for the geodesic distance (see Theorem~\ref{T:corvengenceLipConst}). This is non-trivial, as in general one would only have the semicontinuity of the Lipschitz constant, and exactly in the proof of this fact the continuity of the vector field plays a key role. 

The actual proof is in fact much more delicate. As we show in Lemma~\ref{l:compar}, for a function which is globally defined and Lipschitz continuous for the Euclidean distance, directional Lipschitz continuity in the provisional sense above (that is, along integral curves in the regular Lagrangian flow) is equivalent to Lipschitz continuity of the function when composed with {\em any} integral curve (not necessarily from the selection provided by the regular Lagrangian flow), and even more, when composed with any forward-backward integral curve (that is, Lipschitz curves that can travel forward in time, with velocity given by the vector field, or backward in time, with velocity given by minus the vector field, see Definition~\ref{d:fb}). In fact, the limit degenerate distance provided by the interpolation procedure above precisely encodes the Lipschitz continuity along all forward-backward integral curves, see Lemma~\ref{L:equivalence}. 

Henceforth, in order for a directional Lipschitz extension to exist, it is necessary for the function before the extension to be Lipschitz continuous along all such forward-backward integral curves that intersect the flow-tube domain of the function. This would follow from the triviality (as in Definition~\ref{d:fb}) of forward-backward integral curves for almost every initial point, that is, from the condition that the image of a foward-backward integral curve is contained in the image of a (standard) integral curve. As we are currently not able to prove the almost-everywhere triviality of forward-backward curves, we have to assume it in our main result. 

This fact was overseen in our announcement in~\cite{CC}, in which we claimed a stronger version of our main result, version which is shown to be incorrect by the example in~\cite{MS2}.

\subsection*{Acknowledgments} This work was started during a visit of LC at the University of Basel and carried on during a visit of GC at the University of Padova as a Visiting Scientist. The authors gratefully acknowledge the support and the hospitality of both institutions. LC is a member of the Gruppo Nazionale per l'Analisi Matematica, la Probabilit\`a e le loro Applicazioni (GNAMPA) of the Istituto Nazionale di Alta Matematica (INdAM). GC is partially supported by the ERC Starting Grant 676675 FLIRT. The authors gratefully acknowledge several useful discussions with P.~Bonicatto, P.-E.~Jabin, E.~Marconi, S.~Modena, R.~Monti, L.~Sz\'ekelyhidi, and D.~Vittone.

\section{Integral curves, multi-flows, and a monotone family of distances}
\label{S:not}

Since our setting and argument have a geometric character, we work in~\S\S~\ref{S:not}--\ref{S:directionalV} in the Euclidean\break $N$-dimensional space. However, the first space coordinate plays a special role, therefore we use the notation 
$$
\boxx=(x_{0},\widehat\boxx) \in \R^N = \R \times \R^n\,, 
\text{ where $x_{0}\in\R$ and $\widehat\boxx\in\R^{N-1}=\R^{n}$.}
$$ 
We consider a bounded vector field 
$$
\bov = (v_{0},\widehat\bov) : [0,T]\times\R^n\subset\R^{N}\to\R^{N}\,,
\text{ where $v_{0}\in\R$ and $\widehat\bov\in\R^{N-1}=\R^{n}$}
$$
and we assume that 
$$
\text{the first component of $\bov$ is identically $1$, that is, $v_{0}\equiv1$ with our notation.}
$$
As customary, we say that $\bog : [0,S_{\bog}] \to [0,T]\times\R^n$ is an integral curve of $\bov$ if
$$
\dot \bog(s) =\bov (\bog(s)) \qquad \text{in $\mathcal D' ((0,S_{\bog}))$.}
$$
Since we are assuming that $\bov$ is bounded, integral curves of $\bov$ are Lipschitz continuous.

\begin{definition}
\label{d:fb}
We say that $\bog :  [0,S_{\bog}]  \to [0,T]\times\R^n$ is a forward-backward integral curve of $\bov$ if it is Lipschitz continuous and
\begin{equation}\label{e:d:fb}
\dot \bog(s) =\bov (\bog(s)) \qquad \text{or} \qquad \dot \bog(s) =-\bov (\bog(s))
\qquad\text{for $\Ll^{1}$-a.e.~$s \in  [0,S_{\bog}] $.}
\end{equation}
We say that a forward-backward integral curve $\bog :  [0,S_{\bog}]  \to [0,T]\times\R^n$ is trivial
if there exists an integral curve $\widetilde\bog :  [0,S_{\widetilde\bog}]  \to [0,T]\times\R^n$
such that $\Ima \bog \subset \Ima \widetilde\bog$.
\end{definition}
By the area formula, the image of a negligible set via the map $\gamma_0 : [0,S_{\bog}]  \to [0,T]$ is
negligible. Therefore, the definition in~\eqref{e:d:fb} is well-posed. 
We moreover observe that, if $\bog : \Ig{ }= [0,S_{\bog}]  \to [0,T]\times\R^n$ is a forward-backward integral curve of $\bov$ 
with $\bog(0)=\boxx$ and $\bog(S_{\bog})=\boy$, then
\[
 \boy^{ }- \boxx^{ }
 = \int_{\Ig{ }} \dot\bog(s)  ds  = \int_{\Ig{ }} \dot\gamma_{0}^{ }(s) \bov(\bog^{ }(s))ds  \ .
\]

We also introduce a continuity assumption that we will require on the vector field for most of our later statements and proofs. 

\begin{definition}
\label{D:cont}
We say that a Borel vector field $\bov:[0,T]\times\R^n\to\R^{N}$ is \emph{$\widehat\boxx$-uniformly continuous} if $\bov$ is uniformly continuous in the variable $\widehat\boxx$ with a modulus of continuity which is uniform in $x_{0}$, namely
\[
\forall \rho>0
\qquad
\exists \delta=\delta(\rho)>0
\quad: \quad\text{for $\Ll^{1}$-a.e.~$x_{0}$}\quad
|\bov(x_{0},\widehat\boxx )-\bov(x_{0},\widehat\boxx')|\leq \rho
\quad\text{if $|\widehat\boxx-\widehat\boxx'|\leq \delta$,}
\] 
where $(x_0,\widehat\boxx), (x_0,\widehat\boxx') \in [0,T]\times\R^n$. 
We call the value $\delta=\delta(\rho)$ \emph{modulus of uniform continuity} associated to $\rho>0$.
\end{definition}

\begin{definition}[Multi-flow, flow-tube]
\label{d:multiflow}
Consider a Borel map $\Zz : \TTT\times \Amf\to  \TTT\times \R^{n} $, where~$\Amf$ is a space of parameters, and a Borel vector field $\bov: [0,T]\times\R^n\to\R^{N}$.
We say that $\Zz$ is a \emph{multi-flow} of $\bov$ if the $0$-component of $\Zz( s,\amf)$ is precisely $s$, that is
\[
Z_0(s,\amf)=s\qquad \text{for every~$\amf\in\Amf$ and every $0\leq s\leq T$,}
\]
and moreover
\[
\ddz \Zz( s,\amf)=\bov(\Zz(s,\amf))
\qquad \text{in $\mathcal D' ((0,T))$ for every~$\amf\in\Amf$.}
\]
Given a multi-flow $\Zz$, we call a set $\mcK\subset\R^{N}$ 
a \emph{flow-tube} (associated to $\bov$) if there is $A\subset\Amf$ such that
\[
\mcK= \Zz \left( \TTT \times A \right) \subset \TTT\times\R^{n} \,. 
\]
Throughout this paper, for simplicity we assume that \emph{the image of $\Zz$ is dense in $\TTT\times\R^{n}$}.
\end{definition}

The density of the image of $\Zz$ is a natural assumption in the setting of our application, since we will
consider as multi-flow the regular Lagrangian flow of a vector field defined on $[0,T]\times\R^n$. However, 
for some of our results below we do not need this density assumption (see for instance the comments preceding Lemma~\ref{L:distanceisd}) and we can moreover assume the vector field to be defined on a subset of $[0,T]\times\R^n$ (by making suitable changes to statements and proofs).

\subsection{A family of distances}
\label{Ss:distances}
Given a multi-flow $\Zz$ of a vector field $\bov$, we introduce a family $d_{\lambda}$ of distances on $[0,T]\times\R^{n}$, where the parameter $\lambda\in(0,1]$, and we define their limit $d_{0}$.

\begin{definition}[Competitors]
\label{D:competitors}
Let $\Zz:\TTT\times \Amf\to \TTT\times \R^{n}$ be a {multi-flow} of a vector field $\bov$.
We define \emph{competitor} joining $\boxx$ and $\boy$ a Lipschitz curve $\bog:\Ig{} \to\TTT\times\R^{n}$ defined on 
\[
\Ig{}=[0,S_{\bog}]=\bigcup_{i=0}^{I} [S_{\bog}^{i},S_{\bog}^{i+1}]
\quad\text{where} \quad S_{\bog}^{0}=0 \leq \dots\leq  S_{\bog}^{I+1}=S_{\bog}\ ,
\]
such that 
\[
\text{either\quad$\bog(0)=\boxx$ \;and\; $\bog(S_{\bog})=\boy$\qquad or\qquad$\bog(S_{\bog})=\boxx$ \;and\; $\bog(0)=\boy$}
 \]
and such that for every sub-interval $[S_{\bog}^{i},S_{\bog}^{i+1}]$ 
\begin{enumerate}
\item either $\gamma_{0}$ is constant on the sub-interval $[S_{\bog}^{i},S_{\bog}^{i+1}]$ and $|\dot\bog|\leq \norm{\bov}_{\infty}$,
\item or $\bog(s)=\Zz(\gamma_{0}(s),\amf_{i})$ for all $s\in[S_{\bog}^{i},S_{\bog}^{i+1}]$, for some $\amf_{i}\in\Amf$, and $\dot \gamma_{0} = 1$ in $\mathcal D' ((S_{\bog}^{i},S_{\bog}^{i+1}))$, 
\item  or $\bog(s)=\Zz(\gamma_{0}(s),\amf_{i})$ for all $s\in[S_{\bog}^{i},S_{\bog}^{i+1}]$, for some $\amf_{i}\in\Amf$, and $\dot \gamma_{0} =- 1$ in $\mathcal D' ((S_{\bog}^{i},S_{\bog}^{i+1}))$.
\end{enumerate}
We denote by $\Ignp{}$ the union of the intervals of the first kind, and by $\Igp{}$ the union of the intervals of the second and third kind, so that $\Ig{}=\Ignp{}\cup\Igp{}$.\end{definition}

In general, the class of all competitors joining $\boxx$ and $\boy$ is not closed under local uniform convergence.

\begin{definition}[Distances]
\label{D:distances}
For $0<\lambda\leq 1$ we define the distance $d_{\lambda}  : \left(\TTT \times \R ^{n}\right)^{2} \to [0,+\infty)$ by
\begin{align}
\label{E:distances}
&d_{\lambda}(\boxx,\boy)= \inf \Big\{\mass{\bog}\ | \ \text{$\bog$  competitor joining $\boxx$ and $\boy$}\Big\},
\\
\label{E:masses}
&\mass{\bog}=\Ll^{1}\left(\Igp{}\right) + \frac{1}{\lambda}\,\Ll^{1}\left(\Ignp{}\right)
\ .
\end{align}
Lemma~\ref{L:fwegeg} below guarantees that we can define
$d_0:\left(\TTT \times \R ^{n}\right)^{2}  \to [0,+\infty]$ as
\begin{equation*}
d_{0}(\boxx,\boy)=\lim_{\lambda\downarrow0}d_{\lambda}(\boxx,\boy)=\sup_{0<\lambda\leq 1}d_{\lambda}(\boxx,\boy)\  ,
\end{equation*}
and $d_{0}$ turns out to be a possibly degenerate distance on $[0,T]\times\R^n$. 
\end{definition}

We prove in Lemmas~\ref{L:distanceisd} and~\ref{L:fwegeg} below that the 
distances $d_{\lambda}$ introduced above are well defined lower semicontinuous distances for all $0 \leq \lambda \leq 1$. Moreover, we prove in Lemmas~\ref{L:eqDist1} and~\ref{L:equivalenceDist} below the following characterizations of the distance $d_0$:
\begin{align*}
 d_{0}&(\boxx,\boy)
 \equiv 
  \lim_{{k\to+\infty}}\inf\Big\{\Ll^{1}\left(\Igp{}\right)\ | \ \bog\text{ competitor joining $\boxx$ and $\boy$ with $\Ll^{1}(\Ignp{})<\sfrac{1}{k}$}\Big\} 
  \\
&\equiv 
  \min\Big\{\overline S\ | \ \text{there is a forward-backward integral curve $\overline\bog:[0,\overline S]\to [0,T]\times\R^n$ of $\bov$ joining $\boxx$ and $\boy$}\Big\} 
  \ .
\end{align*}

In particular, $d_{0}(\boxx,\boy)= +\infty $ if the infimum in the first characterization is
taken over an empty set after some $\overline k\in\N$, which happens when $  \Ll^{1}(\Ignp{})>c>0$ for all competitors $\bog$ joining $\boxx$ and $\boy$.
This corresponds in the second characterization to the case when no forward-backward integral curve of $\bov$ joins $\boxx$ and $\boy$.

\medskip 

The next lemma establishes the equivalence between the distances $d_{\lambda}$ with $0<\lambda\leq 1$ and the Euclidean distance, with an upper bound in the equivalence which degenerates as $\lambda\downarrow 0$. We exploit the density of $\Ima \Zz$ in order to have enough competitors in the definition of $d_{\lambda}(\boxx,\boy)$, a property which is necessary for the upper bound in~\eqref{E:rgqgrqg}. 

\begin{lemma}
\label{L:distanceisd}
The function $d_{\lambda}$ is a distance for all $0<\lambda\leq1$ and
\begin{equation}
\label{E:rgqgrqg}
\frac{|\boy-\boxx|}{ \norm{\bov}_{\infty} } \leq   d_{\lambda}(\boxx,\boy) \leq   \frac{ {3} }{\lambda} \left|\boy-\boxx\right|
\qquad
\forall \, \boxx, \boy \in \TTT\times \R^n \,.
\end{equation}
In particular, $d_{\lambda}$ is a continuous function and is equivalent to the Euclidean distance.
\end{lemma}

\begin{proof}
The distance $d_{\lambda}$ is real valued since a competitor joining $\boxx$ and $\boy$ is given by a suitable parameterization of the path composed, for any $\amf\in\Amf$, by 
the segment joining $\boxx$ and $\Zz(x_{0},\amf)$, followed by $\Zz(s,\amf)$ for $s$ between $x_{0}$ and $y_{0}$, and 
finally the segment joining $\boy$ and $\Zz(y_{0},\amf)$.

Let us prove that $d_{\lambda}$ satisfies the triangular inequality.
Let $\boxx^{1},\,\boxx^{2},\,\boxx^{3}\in[0,T]\times\R^{n}$ and $\varepsilon>0$.
Definition~\ref{E:distances} guarantees for all $\varepsilon>0$ the existence of
competitors $\bog^{12}:[0,S_{1}]\to[0,T]\times\R^{n}$ and $\bog^{23}:[0,S_{2}]\to[0,T]\times\R^{n}$ such that
\[
d_{\lambda}(\boxx^{1},\boxx^{2})>\mass{\bog^{12}} -\varepsilon
\quad\text{and}\quad
d_{\lambda}(\boxx^{2},\boxx^{3})>\mass{\bog^{23}}-\varepsilon
\ .
\]
Up to exchanging $\bog^{12}(s)$ with $\bog^{12}(S_{1}-s)$ we can assume that $\bog^{12}(0)=\boxx$ and $\bog^{12}(S_{1})=\boy$, and similarly that $\bog^{23}(0)=\boy$ and $\bog^{23}(S_{2})=\boz$.
Therefore, the curve
\[
\bog(s)=\begin{cases}
\bog^{12}(s) &\text{if $0\leq s\leq S_{1}$}
\\
\bog^{23}(s-S_{1}) &\text{if $S_{1}\leq s\leq S_{2}$}
\end{cases}
\]
is a competitor for $d_{\lambda}(\boxx^{1},\boxx^{3})$ and thus
\[
d_{\lambda}(\boxx^{1},\boxx^{3})\leq \mass{\bog}=\mass{\bog^{12}}+\mass{\bog^{23}}<d_{\lambda}(\boxx^{1},\boxx^{2})+d_{\lambda}(\boxx^{2},\boxx^{3})+2\varepsilon \ .
\]
As $\varepsilon>0$ is chosen arbitrarily this proves that $d_\lambda$ satisfies the triangular inequality.

Non-negativity, symmetry, identity of indiscernibles are clear, therefore $d_{\lambda}$ is a distance. 

The fact that the distance $d_{\lambda}$ is equivalent to the Euclidean distance 
and the continuity as a function of two variables
follow by estimate~\eqref{E:rgqgrqg},
which we now show.
We prove separately the estimates from above and from below.
We first show the estimate from below. By Definition~\ref{D:distances} each competitor $\bog$ satisfies $|\dot\bog|\leq \norm{\bov}_{\infty}$. Integrating this inequality and using that $0<\lambda\leq 1$ we get
\[
|\boy-\boxx| \leq \int_{I_{\bog}}|\dot\bog| \leq  \norm{\bov}_{\infty} \Ll^{1}(I_{\bog})\leq \norm{\bov}_{\infty}\mass{\bog} \ .
\]
Taking the infimum over competitors $\bog$ we obtain $ |\boy-\boxx| \leq \norm{\bov}_{\infty}d_{\lambda}(\boxx,\boy)$.
Regarding the estimate from above, we need the density of the image of $\Zz$. 
Let $\varepsilon>0$ and consider any point $\Zz(x_{0},\amf)$ such that $|\Zz(x_{0},\amf)-\boxx|\leq \varepsilon$, which is possible by the density assumption. A competitor is provided by the path composed by the horizontal segment joining $ \Zz(x_{0},\amf)$ and $\boxx$, at speed $\norm{\bov}_{\infty}$, then following $\Zz(x_{0}+s,\amf)$ for $s\in[0,y_{0}-x_{0}]$ and finally by the horizontal segment joining $\Zz(y_{0},\amf)$ and $\boy$, at speed $\norm{\bov}_{\infty}$: if e.g.~we are considering $y_{0}\geq x_{0}$ then
\[
 \bog (s)=
\begin{cases}
 \boxx+ \frac{s}{S_{\bog }^{1}} ( \Zz(x_{0},\amf)-\boxx)    &\text{if $0\leq s\leq S_{\bog }^{1} :=\sfrac{| \Zz(x_{0},\amf)-\boxx |}{\norm{\bov}_{\infty}}$}
\\
\Zz(x_{0}+s-S_{\bog}^{1},\amf) &\text{if $ S_{\bog}^{1} \leq s\leq  S_{\bog}^{2}:=S_{\bog }^{1} +|y_{0} -x_{0}|$}
\\
\Zz(y_{0} ,\amf) + \frac{s-S_{\bog }^{2}}{S_{\bog^{ }} -S_{\bog }^{2}} ( \boy-\Zz(y_{0} ,\amf))
 &\text{if $S_{\bog}^{2}\leq s\leq S_{\bog }:=S_{\bog}^{2}+ \sfrac{| \boy-\Zz(y_{0} ,\amf) |}{\norm{\bov}_{\infty}}$ .}
\end{cases}
\]
The explicit computation on this competitor proves the upper bound
\label{item:equivalenceEuclidean}
\begin{align*}
d_{\lambda}(\boxx,\boy) &\leq
 |y_{0}-x_{0}|
+\frac{| \Zz(x_{0},\amf)-\boxx |+| \boy-\Zz(y_{0} ,\amf)|}{\lambda\norm{\bov}_{\infty}}
 \\
 &\leq |y_{0}-x_{0}|+\frac{1}{\lambda\norm{\bov}_{\infty}} \left(2\varepsilon+\norm{\bov}_{\infty} |y_{0}-x_{0}|+ |\widehat\boy-\widehat\boxx|\right)
\end{align*}
since $| \boy-\Zz(y_{0} ,\amf)|\leq| \widehat\boy- \widehat\boxx |+|\widehat\Zz(x_{0} ,\amf)-\widehat\Zz(y_{0} ,\amf)|+|\widehat\boxx-\widehat\Zz(x_{0} ,\amf)|\leq  |\widehat\boy-\widehat\boxx|+\norm{\bov}_{\infty} |y_{0}-x_{0}|+\varepsilon$. Thus
\begin{align*}
d_{\lambda}(\boxx,\boy) \leq  \frac{ {3}}{\lambda} \left|\boy-\boxx\right|
\end{align*}
as $\varepsilon$ is arbitrarily small and  $\norm{\bov}_{\infty}\geq 1$.
\end{proof}

We now state and prove simple properties of the distances $d_{\lambda}$ for $0<\lambda\leq 1$.
We also prove that the definition of $d_0$ in Definition~\ref{D:distances} is well posed and 
we show some properties of $d_0$.

\begin{lemma}
\label{L:fwegeg}
The distances in Definition~\ref{D:distances} satisfy the following properties.
\begin{enumerate}
\item 
\label{ite:distanceconvergence}
The distances $d_{\lambda}$ increase monotonically when $\lambda\downarrow 0$. We can therefore define 
\begin{equation*}
d_{0}(\boxx,\boy)=\lim_{\lambda\downarrow0}d_{\lambda}(\boxx,\boy)=\sup_{0<\lambda\leq 1}d_{\lambda}(\boxx,\boy)\ ,\qquad d_{0}\in [0,+\infty].
\end{equation*}
Moreover, $d_{0}$ is a lower semicontinuous distance, which is possibly degenerate.
\item 
\label{ite:qqreoihr}
For $0\leq\lambda\leq 1$, the inequality $\quad\left|y_{0}-x_{0}\right| \leq  d_{\lambda}(\boxx,\boy)\quad$ holds.
\item
\label{ite:smallchar}
If $\boxx=\Zz(x_{0},\amf)$ and $\boy=\Zz(y_{0},\amf)$ for some $\amf\in\Amf$ we have
\begin{align*}
 &d_{0}(\boxx,\boy)
 = d_{\lambda}(\boxx,\boy) =
\left|y_{0}-x_{0}\right| 
\qquad \text{ for all $0<\lambda\leq 1$.}
\end{align*}
\end{enumerate}
\end{lemma}
\begin{proof}
Let us consider Property~\eqref{ite:distanceconvergence}.
Since the family competitors joining $\boxx$ and $\boy$ is the same independently of $\lambda$, but the weight increses as $\lambda \downarrow 0$ by~\eqref{E:distances} and~\eqref{E:masses}, we see that $d_{\lambda}(\boxx,\boy)$ increases as $\lambda\downarrow0$ and thus it admits a pointwise limit.  
Being the supremum of lower semicontinuous functions by Lemma~\ref{L:distanceisd}, the function $d_{0}$ is lower semicontinuous.
Moreover, it is a distance being a monotone increasing pointwise limit of distances (recall again Lemma~\ref{L:distanceisd}).

Property~\eqref{ite:qqreoihr} is a consequence of the fact that each 
competitor $\bog$ for $d_{\lambda}(\boxx,\boy)$
satisfies $\dot\gamma_{0}\in\{0,\pm1\}$. Indeed, together with $0<\lambda\leq 1$ this implies
\begin{equation}
\label{E:gegr}
|y_{0}-x_{0}|=\left| \int_{\Ig{}}\dot\gamma_{0} \right| \leq 
 \int_{\Ig{}}\left|\dot\gamma_{0} \right|\leq \Ll^{1}\left(\Ig{}\right)\leq \mass{\bog}\, ,
\end{equation}
and property~\eqref{ite:qqreoihr} follows by taking the infimum~\eqref{E:distances} over admissible competitors.

Concerning Property~\eqref{ite:smallchar}, the path $\Zz(s,\amf)$ joining $\boxx$ and $\boy$ is an admissible competitor for $d_{\lambda}(\boxx,\boy)$ and it satisfies $\Ig{}=\Igp{}$.
Therefore,~\eqref{E:gegr} becomes a chain of equalities for $\lambda\in(0,1]$. Using Property~\eqref{ite:distanceconvergence} this equality also holds for $\lambda={0}$.
\end{proof}

We now provide the first characterization of $d_{0}$.

\begin{lemma}
\label{L:eqDist1}
For every $\boxx,\boy \in [0,T]\times\R^n$ we have
\begin{align}
\label{E:rfienjgegrbbeqb}
 &d_{0}(\boxx,\boy)
 \equiv 
  \lim_{{k\to+\infty}}\inf\Big\{\Ll^{1}\left(\Igp{}\right)\ | \ \bog\text{ competitor joining $\boxx$ and $\boy$ with $\Ll^{1}(\Ignp{})<\sfrac{1}{k}$}\Big\} \ .
\end{align}
\end{lemma}

\begin{proof}
Consider a sequence $\bog^{\lambda}$ of competitors joining $\boxx$ and $\boy$ 
and such that
\[ \lim_{\lambda\to 0} \mass{\bog^{\lambda}}= \sup_{0<\lambda\leq 1}d_{\lambda}(\boxx,\boy) = d_0(\boxx,\boy) \ . 
\]
Two cases are possible, and we will prove the claimed characterization in both of them.

\noindent\underline{\sc Case 1.} If $\Ll^{1}(\Ignp{\lambda})\to 0$, up to subsequences, then we estimate from above the infimum in the right hand 
side of~\eqref{E:rfienjgegrbbeqb} by using the sequence~$\bog^{\lambda}$: this yields the ``$\geq$'' inequality
in~\eqref{E:rfienjgegrbbeqb}, since
\[
\lim_{k\to+\infty}\inf\left\{\Ll^{1}\left(\Igp{}\right)\ | \ \Ll^{1}(\Ignp{})<\sfrac{1}{k}\right\} \leq \liminf_{\lambda\to 0}\Ll^{1}\left(\Igp{\lambda}\right)
\leq  \liminf_{\lambda\to 0} \mass{\bog^{\lambda}} =\sup_{0<\lambda\leq 1}d_{\lambda}(\boxx,\boy) 
= d_0(\boxx,\boy) \ .
\]
On the other hand, if $\bog^{k}$ is any sequence of competitors joining $\boxx$ and $\boy$ with $\Ll^{1}(\Ignp{k})<\sfrac{1}{k}$ we have
\[
d_{\lambda}(\boxx,\boy) \leq \mass{\bog^{k}}=\Ll^{1}\left(\Igp{k}\right) + \frac{1}{\lambda}\,\Ll^{1}\left(\Ignp{k}\right)
\leq \Ll^{1}\left(\Igp{k}\right) + \frac{1}{k\lambda} \ ,
\]
which implies $d_{\lambda}(\boxx,\boy) \leq\liminf_{k\to\infty} \Ll^{1}\left(\Igp{k}\right) $,
and therefore
\[
d_{\lambda}(\boxx,\boy) \leq \lim_{{k\to+\infty}}\inf\Big\{\Ll^{1}\left(\Igp{}\right)\ | \ \bog\text{ competitor joining $\boxx$ and $\boy$ with $\Ll^{1}(\Ignp{})<\sfrac{1}{k}$}\Big\} \qquad\forall\lambda\in(0,1] \ .
\] 
When $\lambda\downarrow0$, by definition of $d_0$ we obtain the same inequality for $d_{0}$. This proves the ``$\leq$'' inequality in~\eqref{E:rfienjgegrbbeqb} and hence it shows the equality. 

\noindent\underline{\sc Case 2.} 
Suppose now that there is no subsequence $\lambda\downarrow0$ for which $\Ll^{1}(\Ignp{\lambda})\to 0$: then $\Ll^{1}(\Ignp{\lambda})>c>0$ for all $\lambda$.
This implies $\mass{\bog^{\lambda}}>\sfrac{c}{\lambda}$ for all $\lambda$ and therefore by~\eqref{E:distances}
and~\eqref{E:masses} 
\begin{equation}\label{E:rgrgrgrrgq}
d_{0}(\boxx,\boy)=\sup_{0<\lambda\leq 1}d_{\lambda}(\boxx,\boy)=\lim_{\lambda\downarrow 0} \mass{\bog^{\lambda}}=+\infty\ .
\end{equation}
We now show by contradiction that also the right hand side of~\eqref{E:rfienjgegrbbeqb} has the value $+\infty$.
Indeed, the right hand side~\eqref{E:rfienjgegrbbeqb} is finite only if there exists a sequence of competitors $\bog^{k}$ with $\Ll^{1}(\Ignp{k})\leq \sfrac{1}{k}$ and $\Ll^{1}(\Igp{k})$ bounded.
If this was the case, using
\[
d_{\lambda}(\boxx,\boy)\leq \mass{\bog^{k}}  \quad \forall k
\quad\quad\text{by~\eqref{E:distances}}
\]
and being $\Ll^{1}(\Ignp{k})/\lambda\leq \sfrac{1}{\lambda k}$ vanishingly small, we would have by~\eqref{E:masses}
\[
d_{\lambda}(\boxx,\boy)\leq \liminf_{k\to\infty} \mass{\bog^{k}}= \liminf_{k\to\infty}  \Ll^{1}(\Igp{k})
\qquad\text{bounded,}
\]
but this is not possible because we have already shown that $\sup_{\lambda}d_{\lambda}(\boxx,\boy)=+\infty$ in~\eqref{E:rgrgrgrrgq}.
\end{proof}

\begin{remark}
If the multi-flux was ``saturated'', in the sense that uniform limits of sequences of integral curves $\Zz(\cdot,\amf_{k})$ correspond to some integral curve $\Zz(\cdot,\amf)$, then in~\eqref{E:rfienjgegrbbeqb} we would have the limit of minima instead of the limit of infima, but we would not have other relevant differences.
In particular, for every $0 \leq \lambda\leq 1$ the distances $d_{\lambda}$ constructed by $\Zz$ turn out to be the same as the ones constructed by the ``saturation'' of $\Zz$.
 \end{remark}
 
We now show the second characterization of the distance $d_{0}$. Its proof is based on a compactness result 
for forward-backward integral curves (see Lemma~\ref{L:trgeibonwn}), which we need for the proof of one
of the inequalities in the characterization.
 
\begin{lemma}
\label{L:equivalenceDist}
Suppose $\bov$ is $\widehat\boxx$-uniformly continuous as in Definition~\ref{D:cont}.
Let $\Zz$ be a multi-flow of $\bov$ as in Definition~\ref{d:multiflow}. Then for every $\boxx,\boy \in [0,T]\times\R^n$ we have
\begin{equation}\label{e:ch2}
d_{0}(\boxx,\boy)\equiv 
\begin{cases}
+\infty \qquad\text{if there is no forward-backward integral curve of $\bov$ joining $\boxx$ and $\boy$}
\\
\min\left\{ {\overline S\ \Bigg| \ 
	\begin{split}\text{there is a forward-backward integral curve $\overline\bog:[0,\overline S]\to [0,T]\times\R^n$}\\\text{of $\bov$ joining $\boxx$ and $\boy$}\end{split}}\right\} \ .
\end{cases}
\end{equation}
\end{lemma}

\begin{proof}[Proof of Lemma~\ref{L:equivalenceDist}, first part]
We show that, if there exists a forward-backward integral curve of $\bov$ defined on $[0,\overline S]$ and joining $\boxx$ and $\boy$, then $d_{0}(\boxx,\boy)\leq \overline S$. This proves that
\begin{equation}\label{e:ch2part}
d_{0}(\boxx,\boy) \leq 
\begin{cases}
+\infty \qquad\text{if there is no forward-backward integral curve of $\bov$ joining $\boxx$ and $\boy$}
\\
\inf\left\{ {\overline S\ \Bigg| \ 
	\begin{split}\text{there is a forward-backward integral curve $\overline\bog:[0,\overline S]\to [0,T]\times\R^n$}\\\text{of $\bov$ joining $\boxx$ and $\boy$}\end{split}}\right\}  \ . 
\end{cases}
\end{equation}
Notice that~\eqref{e:ch2part} is trivial if there is no forward-backward integral curve of $\bov $ joining $\boxx$ and $\boy$. 

Let $\overline S$ be any real value for which there exists a forward-backward integral curve $\overline\bog:[0, \overline S]\to [0,T]\times\R^{n}$ of $\bov$ joining $\boxx$ and $\boy$.
Let $k\in\N$.
We construct competitors $\bog^{k}$ joining $\boxx$ and $\boy$ satisfying 
\begin{align}
\label{E:rggpgrggeqgeq}
\Ll^{1}(\Ignp{k})<\sfrac{1}{k}
\quad\text{and}\quad
\liminf_{k\to\infty}\Ll^{1}(\Igp{k})\leq \overline S \ .
\end{align}
This will be enough to conclude in view of the characterization in~\eqref{E:rfienjgegrbbeqb}.

By the density assumption on $\Ima\Zz$, there are $\amf^{k,1},\dots,\amf^{k,k}$ such that 
\[
\left|\Zz\left(\overline\gamma_{0}\left(i\frac{\overline S}{k}\right),\amf^{k,i}\right) -\overline\bog\left(i\frac{\overline S}{k}\right)\right|\leq \frac{1}{k^{2}} 
\qquad i=1,\dots,k \ .
\]
Define thus $t^{k,i}:=\overline\gamma_{0}\left(i\frac{\overline S}{k}\right)$ for $i=0,1,\dots,k$ and
\begin{align*}
&S_{\bog^{k}}^{0}:=0\ ,
\qquad S_{\bog^{k}}^{1}:=\frac{\left|\Zz\left(\overline\gamma_{0}(0),\amf^{k,1}\right)-\overline\bog(0)\right|}{\norm{\bov}_{\infty}}
\ ,\qquad
S_{\bog^{k}}:=S_{\bog^{k}}^{2k+1}
:= S_{\bog^{k}}^{2k} + \frac{\left|\Zz\left(\overline\gamma_{0}(\overline S),\amf^{k,k}\right)-\overline\bog(\overline S)\right|}{\norm{\bov}_{\infty}}\,,
\\
&S_{\bog^{k}}^{2i}=S_{\bog^{k}}^{2i-1}+\left|t^{k,i}-t^{k,i-1}\right|\qquad\qquad \text{for $i=1,\dots,k$,}
\\
&S_{\bog^{k}}^{2i+1}=S_{\bog^{k}}^{2i}+\frac{\left|\Zz\left(t^{k,i},\amf^{k,i}\right)-\Zz\left(t^{k,i},\amf^{k,i+1}\right)\right|}{\norm{\bov}_{\infty}}
\qquad\qquad \text{for $i=1,\dots,k-1$,}
\end{align*}
and admissible competitors $\bog^{k}:[0,S_{\bog^{k}}]\to\TTT\times\R^{n}$ by
\[
\bog^{k}(s)
=\begin{cases}
\overline\bog(0)+\frac{s}{S_{\bog^{k}}^{1}}\left(\Zz\left(\overline\gamma_{0}(\overline 0),\amf^{k, 1}\right) -\overline\bog(0)\right)
&0\leq s\leq S_{\bog^{k}}^{1}
\\
\Zz\left(t^{k,i-1}+s,\amf^{k,i}\right) & S_{\bog^{k}}^{2i-1}\leq s\leq S_{\bog^{k}}^{2i} , \text{ if $t^{k,i-1}\leq t^{k,i}$}\\
\Zz\left(t^{k,i-1}-s,\amf^{k,i}\right) & S_{\bog^{k}}^{2i-1}\leq s\leq S_{\bog^{k}}^{2i} , \text{ if $t^{k,i-1}\geq t^{k,i}$}
\\
\Zz\left(t^{k,i},\amf^{k,i}\right)
+\frac{s-S_{\bog^{k}}^{2i}}{S_{\bog^{k}}^{2i+1}-S_{\bog^{k}}^{2i}}
\left(\Zz\left(t^{k,i},\amf^{k,i+1}\right)-\Zz\left(t^{k,i},\amf^{k,i}\right)\right)
& S_{\bog^{k}}^{2i }\leq s\leq  S_{\bog^{k}}^{2i+1}
\\
\overline\bog(\overline S)+\frac{s-S_{\bog^{k}}^{2k+1}}{S_{\bog^{k}}^{2k}-S_{\bog^{k}}^{2k+1}}\left(\Zz\left(\overline\gamma_{0}(\overline S) ,\amf^{k, k}\right) -\overline\bog(\overline S)\right)
& S_{\bog^{k}}^{2k}\leq s\leq  S_{\bog^{k}}^{2k+1} \,, 
\end{cases}
\]
where in the definition above $i=1,\dots, k$.
Notice that on each $[S_{\bog^{k}}^{2i-1}, S_{\bog^{k}}^{2i}]$ we follow $\Zz\left(\cdot,\amf^{k,i}\right) $ from $\overline\gamma_{0}\left(\sfrac{i\overline S}{k}\right)$ to $\overline\gamma_{0}\left(\sfrac{(i+1) \overline S}{k}\right)$, while on $[S_{\bog^{k}}^{2i},S_{\bog^{k}}^{2i+1}]$ we make straight horizontal junctions with $|\dot\bog^{k}|=\norm{\bov}_{\infty}$.
It is clear that the competitors $\bog^{k}$ join $\boxx$ and $\boy$. Moreover, they satisfy
\[
\Ll^{1} \left(\Igp{k}\right)=\sum_{i=1}^{k}|t^{k,i}-t^{k,i-1}|=
\sum_{i=1}^{k}\left|\overline\gamma_{0}\left(i\frac{\overline S}{k}\right)-\overline\gamma_{0}\left((i-1)\frac{\overline S}{k}\right)\right|
\leq\overline S
\]
and, as we now show,
\begin{align}
\norm{\bov}_{\infty}\Ll^{1} \left(\Ignp{k}\right)
&=
{\left|\Zz\left(\overline\gamma_{0}(0),\amf^{k,1}\right)-\overline\bog(0)\right|}
+
{\left|\Zz\left(\overline\gamma_{0}(\overline S),\amf^{k,k}\right)-\overline\bog(\overline S)\right|}
\notag\\\notag
&\qquad+
\sum_{i=1}^{k-1}\left|\Zz\left(t^{k,i},\amf^{k,i}\right)-\Zz\left(t^{k,i},\amf^{k,i+1}\right)\right|
\\
\label{E:ooooororroror}&\leq \frac{2}{k^{2}}+ \frac{2+2\norm{\bov}_{\infty}\overline S}{k}+ \rho_{k} {\overline S} \ ,
\end{align}
where we estimated $\left|\Zz\left(\overline\gamma_{0}(\overline S),\amf^{k,k}\right)-\overline\bog(\overline S)\right|\leq \sfrac{1}{k^2}$ and $\left|\Zz\left(\overline\gamma_{0}(0),\amf^{k,1}\right)-\overline\bog(0)\right|\leq \sfrac{2\norm{\bov}_{\infty}\overline S}{k} + \sfrac{1}{k^2}$, since
\begin{equation}\label{E:ggegegegeeg}
\left|\Zz\left(\overline\gamma_{0}(r),\amf^{k,i}\right)-\overline\bog(r')\right|\leq \left|\Zz\left(t^{k,i},\amf^{k,i}\right)-\overline\bog\left(i\frac{\overline S}{k}\right)\right| + \left|\Zz\left(t^{k,i},\amf^{k,i}\right)-\Zz\left(\overline\gamma_{0}(r),\amf^{k,i}\right)\right|+\left|\overline\bog(r')-\overline\bog\left(i\frac{\overline S}{k}\right)\right|
\end{equation}
for $(i-1)\sfrac{\overline S}{k}\leq r, r'\leq i \sfrac{\overline S}{k} $, for $i=1,\dots,k$, and we introduced
\begin{align}\label{E:grgrirgiingnrngrngrw}
\rho_{k}:=\max\left\{|\bov(t,\widehat \boz)-\bov(t,\widehat \boz')|\ :\ |\widehat\boz-\widehat\boz'|\leq2\norm{\bov}_{\infty} \sfrac{\overline S}{k}+\sfrac{1}{k^{2}}\right\}\quad\downarrow\quad0 \ .
\end{align}

We now prove~\eqref{E:ooooororroror} by estimating the terms of the remaining sum. By the triangular inequality
\begin{align}
\left|\Zz\left(t^{k,i},\amf^{k,i}\right)-\Zz\left(t^{k,i },\amf^{k,i+1}\right)\right|
&\leq\left|\Zz\left(t^{k,i},\amf^{k,i}\right)-\overline\bog\left(i\frac{\overline S}{k}\right)\right|
+\left|\overline\bog\left((i+1)\frac{\overline S}{k}\right)-\Zz\left(t^{k,i+1},\amf^{k,i+1}\right)\right|\notag
\\ \label{E:grgrreggerreg}
&\quad +\left|
\overline\bog\left((i+1)\frac{\overline S}{k}\right)-\overline\bog\left(i\frac{\overline S}{k}\right)+\Zz\left(t^{k,i},\amf^{k,i+1}\right)-\Zz\left(t^{k,i+1},\amf^{k,i+1}\right) \right| \ .
\end{align}
Each one of the first two addends is less than $\sfrac{1}{k^2}$ by construction; moreover 
\begin{align*}
\overline\bog\left((i+1)\frac{\overline S}{k}\right)&-\overline\bog\left(i\frac{\overline S}{k}\right) +\Zz\left(t^{k,i},\amf^{k,i+1}\right)-\Zz\left(t^{k,i+1},\amf^{k,i+1}\right)
\\& =
\int_{i\sfrac{\overline S}{k}}^{(i+1)\sfrac{\overline S}{k}}  \dot{\overline\bog}  (r) dr  
 -\int_{t^{k,i}}^{t^{k,i+1}}\!\!\!\!\!\bov\left(\Zz\left(s,\amf^{k,i+1}\right)\right)ds
\\& =
\int_{i\sfrac{\overline S}{k}}^{(i+1)\sfrac{\overline S}{k}}  \dot{\overline\gamma}_{0} (r)\bov\left( \overline\bog\left(r \right)\right)dr  
 -\int_{t^{k,i}}^{t^{k,i+1}}\!\!\!\!\!\bov\left(\Zz\left(s,\amf^{k,i+1}\right)\right)ds  \ .
\end{align*}
Consider an auxiliary curve 
$$
\widetilde\bog\ :\ \left[ i\sfrac{\overline S}{k},(i+1)\sfrac{\overline S}{k}+|{t^{k,i+1}}-t^{k,i}|\right]\ \to [0,T] \ \times\R^{n}
$$
which juxtaposes $\overline\bog|_{[i \sfrac{\overline S}{k} , (i+1)\sfrac{\overline S}{k}]}$ and $\Zz\left(\cdot ,\amf^{k,i+1}\right)$ from $t^{k,i+1}$ to $t^{k,i}$, so that the $0$-component ${\widetilde\gamma}_{0}$ of $\widetilde\bog$ is Lipschitz continuous and runs from $t^{k,i}$ to $t^{k,i}$ itself:
$$
\widetilde\bog(r)
=
\begin{cases}
 \overline\bog\left(r \right) & i\sfrac{\overline S}{k}\leq r\leq (i+1)\sfrac{\overline S}{k} \ ,
 \\
 \Zz\left(t^{k,i+1}-\left(r-(i+1)\sfrac{\overline S}{k}\right),\amf^{k,i+1}\right) &     (i+1)\sfrac{\overline S}{k}<r<(i+1)\sfrac{\overline S}{k}+{t^{k,i+1}}-t^{k,i} \text{ or}
\\
 \Zz\left(t^{k,i+1}+\left(r-(i+1)\sfrac{\overline S}{k}\right),\amf^{k,i+1}\right) &     (i+1)\sfrac{\overline S}{k}<r<(i+1)\sfrac{\overline S}{k}+t^{k,i}-{t^{k,i+1}} \ .
\end{cases}
$$
Then one notices by definition of $\widetilde\bog$ that 
\begin{align*}
\int_{i\sfrac{\overline S}{k}}^{(i+1)\sfrac{\overline S}{k}}  \dot{\overline\gamma}_{0} (r)\bov\left( \overline\bog\left(r \right)\right)dr  
-\int_{t^{k,i}}^{t^{k,i+1}}\!\!\!\!\!\bov\left(\Zz\left(s,\amf^{k,i+1}\right)\right)ds
=
\int_{i\sfrac{\overline S}{k}}^{(i+1)\sfrac{\overline S}{k}+|{t^{k,i+1}}-t^{k,i}|}  \dot{\widetilde\gamma}_{0} (r)\bov\left( \widetilde\bog\left(r \right)\right)dr  
\end{align*}
and thus applying the area formula to ${\widetilde\gamma}_{0}$, which satisfies $\dot{\overline\gamma}_{0}\in\{-1;1\}$, one has
\begin{align*}
\int_{i\sfrac{\overline S}{k}}^{(i+1)\sfrac{\overline S}{k}}  \dot{\overline\gamma}_{0} (r)\bov\left( \overline\bog\left(r \right)\right)dr  
-\int_{t^{k,i}}^{t^{k,i+1}}\!\!\!\!\!\bov\left(\Zz\left(s,\amf^{k,i+1}\right)\right)ds 
=
\int_{{\overline\gamma}_{0}([i\frac{\overline S}{k},(i+1)\frac{\overline S}{k}])} 
\sum_{r\in \left(\widetilde\gamma_{0} \right)^{-1}(s)} \dot{\widetilde\gamma}_{0} (r)\bov\left( \widetilde\bog\left(r \right)\right)ds
 \ .
\end{align*}
In the summation within the last integrand, the addends corresponding to each $s\in{\overline\gamma}_{0}([\sfrac{ i\overline S}{k},\sfrac{(i+1) \overline S}{k}])$ can be grouped in couples with opposite coefficient $\dot{\widetilde\gamma}_{0}$; being, as a consequence of~\eqref{E:ggegegegeeg},
\[
 \left|\widehat{\widetilde\bog}\left(r \right)-\widehat{\widetilde\bog}\left(r'\right)\right|\leq2\norm{\bov}_{\infty} \sfrac{\overline S}{k}+\sfrac{1}{k^{2}}\ ,
\]
 so that
\[
\left|\bov\left({\widetilde\bog}\left(r \right)\right)-\bov\left({\widetilde\bog}\left(r'\right)\right)\right|
\leq \rho_{k} \ ,
\] 
by the $\widehat\boxx$-uniform continuity and definition~\eqref{E:grgrirgiingnrngrngrw}, this allows to conclude the estimate of~\eqref{E:grgrreggerreg} as
\begin{align*}
\left|\Zz\left(t^{k,i},\amf^{k,i}\right)-\Zz\left(t^{k,i },\amf^{k,i+1}\right)\right|
&\leq  \frac{2}{k^{2}} +\left| 
\int_{{\overline\gamma}_{0}([i\frac{\overline S}{k},(i+1)\frac{\overline S}{k}])} 
\sum_{r\in \left(\widetilde\gamma_{0} \right)^{-1}(s)} \dot{\widetilde\gamma}_{0} (r)\bov\left( \widetilde\bog\left(r \right)\right)ds
\right|
\\
&\leq \frac{2}{k^{2}} +\rho_{k} \cdot  \frac{\overline S}{k} \ .
\end{align*}
In particular,~\eqref{E:rggpgrggeqgeq} holds.
\end{proof}

For the proof of the opposite inequality in Lemma~\ref{L:equivalenceDist} and of the fact that 
the minimum at the right hand side is attained we need the following 
compactness result, which holds thanks to the $\widehat\boxx$-uniform continuity of~$\bov$. 

\begin{lemma}
\label{L:trgeibonwn}
Suppose $\bov$ is $\widehat\boxx$-uniformly continuous as in Definition~\ref{D:cont}.
Consider a sequence of competitors $\{\bog^{k}\}_{k\in\N}$ joining $\boxx$ and $\boy$ such that $\Ll^{1}(\Ignp{k})\downarrow0$ and $\Ll^{1}(\Igp{k})\downarrow \overline S\in(0,+\infty)$. Then, up to subsequences, $\{\bog^{k}\}_{k\in\N}$ converges uniformly on $[0,\overline S]$ to a forward-backward integral curve of $\bov$, up to re-parameterization.
If moreover $\overline S=d_{0}(\boxx,\boy)$, then $\{\bog^{k}\}_{k\in\N}$ converges to a forward-backward integral curve of $\bov$.
\end{lemma}

\begin{proof}
Fix the ideas on the case when the starting point of each $\bog^{k}$ is $\boxx$ and the final point $\boy$, up to reversing the parameterizations.
We have that $\Ll^{1}(\Ig{k})=\Ll^{1}(\Igp{k})+\Ll^{1}(\Ignp{k})$ decreases to $\overline S$.
Let $\bog:[0,\overline S]\to{\R^{n}}$ denote a uniform limit of $\bog^{k}|_{[0,\overline S]}$, up to subsequences, which exists by Ascoli-Arzel\`a theorem since all curves are $\norm{\bov}_{\infty}$-Lipschitz.
We clearly have that $\bog(\overline S)=\boy$ and $\bog(0)=\boxx$. We now show
\begin{equation}
\label{E:reongeegegeg}
\dot  \bog (\overline t )  =  \dot\gamma_{0} (\overline t)\bov(\bog(\overline t)) 
\qquad \text{in $\D'((0,\overline S))$}
\ .
\end{equation}
This proves the first claim:  the direction of the velocity $\dot  \bog$ of $\bog$ is given by $\bov(\bog ) $.
One could thus  change parametrization so that $|\dot\gamma_{0}|\equiv 1 $.

 Notice first that, since $\bog^{k}$ converges uniformly to $\bog$, then $\dot\bog^{k}$ converges weakly to $\dot\bog$.
We thus just need to show that $\dot\gamma_{0} (\overline t)\bov(\bog(\overline t)) $ is the weak limit of $\dot\bog^{k}$.
By definition of competitor
\[
\dot\bog^{k}(s)
=
\dot\gamma_{0}^{k}(s)\bov(\bog^{k}(s))
\qquad \text{if $s\in \Igp{k}$.}
\]
Being $\dot\gamma_{0}^{k}(s)=0$ on $ \Ignp{k} $ one can also write
\begin{equation}
\label{E:grfrrefr}
\dot\bog^{k}(s)
=
\dot\gamma_{0}^{k}(s)\bov(\bog^{k}(s))+ \dot\bog^{k}(s)\Uno_{\Ignp{k}}(s)
\qquad \text{for all $s\in \Ig{k}$.}
\end{equation}
Moreover, by the specific sequence we are considering $\Ll^{1}(\Ignp{k})\downarrow0$.
In particular \[ \dot\bog^{k}(s)\Uno_{\Ignp{k}}(s)  \quad\text{is weakly converging to the null function.}\]
Since $\bov(\bog^{k}(s))$ is strongly converging to $\bov(\bog (s))$ and $\dot\gamma_{0}^{k}(s)$ is weakly converging to $\dot\gamma_{0} (s)$ in the weak limit of~\eqref{E:grfrrefr} we obtain~\eqref{E:reongeegegeg}.

Under the additional assumption that $\Ll^{1}(\Ignp{k})$ decreases precisely to $\overline S=d_{0}(\boxx,\boy)$ we now show that the uniform
limit $\bog:[0,\overline S]\to{\R^{n}}$ constructed above is directly a forward-backward integral curve of $\bov$,
without need of a reparametrization.

We remind that we have already shown in~\eqref{e:ch2part} that $d_{0}(\boxx,\boy)$ is less or equal than the length $\widetilde S$ of the domain of any forward-backward integral curve $\widetilde\bog:[0, \widetilde S]\to [0,T]\times\R^n$ of $\bov$ joining $\boxx$ and $\boy$.
Suppose by contradiction that the uniform limit $\bog:[0,d_{0}(\boxx,\boy)]\to{[0,T]\times\R^{n}}$ of $\{\bog^{k}\}_{k\in\N}$, constructed above, satisfies~\eqref{E:reongeegegeg} and $|\dot\gamma_{0}(s)|< 1-\varepsilon<1 $ on a set $P\subset[0,d_{0}(\boxx,\boy)]$ of positive measure.
If one changes parameterization so that the reparaterized curve $\widetilde\gamma:[0,\widetilde S]\to{[0,T]\times\R^{n}}$ satisfies $|\dot{\widetilde\gamma}_{0}|\equiv 1 $, then necessarily 
\[
\widetilde S< \left(1-\varepsilon \Ll^{1}(P)\right)\overline S<d_{0}(\boxx,\boy) \,,
\]
which is a contradiction to~\eqref{e:ch2part}.
\end{proof}

\begin{proof}[Proof of Lemma~\ref{L:equivalenceDist}, second part]
Assume $\overline S:=d_{0}(\boxx,\boy)$ is finite.
By the previous characterization~\eqref{E:rfienjgegrbbeqb}, there is a sequence of competitors~$\{\bog^{k}\}_{k\in\N}$   joining $\boxx$ and $\boy$ with $\Ll^{1}(\Ignp{k})<\sfrac{1}{k}$ and $\Ll^{1}(\Igp{k})\downarrow \overline S$.
By the compactness Lemma~\ref{L:trgeibonwn}, up to subsequences $\{\bog^{k}|_{[0,\overline S]}\}_{k\in\N}$ converges uniformly to some Lipschitz continuous curve $\overline\bog:[0,\overline S]\to [0,T]\times\R^{n}$ which is a forward-backward integral curve of $\bov$ joining $\boxx$ and $\boy$.

This shows that 
\begin{equation}\label{e:ch2partbis}
d_{0}(\boxx,\boy) \geq 
\begin{cases}
+\infty \qquad\text{if there is no forward-backward integral curve of $\bov$ joining $\boxx$ and $\boy$}
\\
\inf \Big\{ \overline S\ | \ \text{there is a forward-backward integral curve $\overline\bog:[0,\overline S]\to [0,T]\times\R^n$ of $\bov$ joining $\boxx$ and $\boy$}\Big\} \ . 
\end{cases}
\end{equation}
Notice that~\eqref{e:ch2partbis} is trivial if $d_{0}(\boxx,\boy)=+\infty$. Together with the upper bound~\eqref{e:ch2part}, the lower bound~\ref{e:ch2partbis} implies the equality in~\eqref{e:ch2}, in which
the forward-backward integral curve $\overline\bog:[0,\overline S]\to [0,T]\times\R^{n}$ is an element which realizes the minimum.
\end{proof}

\section{Directional Lipschitz continuity}\label{s:dirlip}
We introduce in the present section the notion of directional Lipschitz continuity.
Formalizing this notion is very intuitive when integral curves of $\bov$ are unique. It is however more complex otherwise: two points $\boxx$ and $\boy$ in the domain could be joined by a concatenation of two integral curves of $\bov$, one run forward and one run backward, which meet in a third point $\boz$ not belonging in the domain (see Figure~\ref{fig:123ffff}). 
\begin{figure}[ht]
\centering
\begin{tikzpicture}[scale=1]
\centering 
\draw [domain=0:1,smooth,variable=\x, very thick,color=purple] plot  (\x,{ (1-\x)^3}) ;
\draw [domain=0:1,smooth,variable=\x, very thick,color=purple] plot  (\x,{ (\x-1)^3});
\draw[->,color=black] (-.2,0) -- (1.5,0) node[right] {$x_0$};
\draw[->,color=black] (0,-1.3) -- (0,1.3) node[right] {$x_1$};
\filldraw[fill=blue!40!white, draw=black] (0,1)circle (0.1cm) node[left] {$\boxx$}; 
\filldraw[fill=blue!40!white, draw=black] (0,-1)circle (0.1cm) node[left] {$\boy$}; 
\filldraw[fill=yellow, draw=black] (1,0)circle (0.1cm) node[below] {$\boz$}; 
\end{tikzpicture}
\caption{Consider two points $\boxx$ and $\boy$ that are not joined by an integral curve of the vector field. Suppose nevertheless they are joined by a concatenation of two integral curves meeting at a point $\boz$ out of the domain of the function $\phi$.
Fix for example $x_{0}=y_{0}=0$ and $z_{0}=1$.
If $\phi$ is $1$-Lipschitz continuous along integral curves then by the triangular inequality necessarily one has the nontrivial constraint $|\phi(\boxx)-\phi(\boy)|\leq 2$. }
\label{fig:123ffff}
\end{figure}
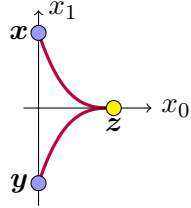
Then there is a constraint on the values of $\phi$ at $\boxx$ and $\boy$, constraint  which is not immediately clear by just looking at values of $\phi$ only along each integral curve of $\bov$.
Remarks~\ref{r:compar00} and~\ref{r:compar} below provide trickier examples.

\medskip

We already constructed in Section~\ref{Ss:distances} a distance $d_{0}$ which encodes the structure of the forward-backward integral curves of $\bov$ joining two point $\boxx$ and $\boy$, as stated in Lemma~\ref{L:equivalenceDist} above.
We thus define the directional Lipschitz continuity exploiting this distance $d_{0}$.

\begin{definition}[Directional Lipschitz continuity]
\label{def:equivalenceGen}
Given a multi-flow $\Zz$ of a vector field $\bov$ and a Borel set $D \subset [0,T]\times\R^n$, we say that a function $\phi: D\to\R$ is \emph{$L$-directionally Lipschitz continuous} if
\begin{equation}
\label{E:LrrgergGeneral0}
|\phi( \boy)-\phi( \boxx)|\leq L d_{0}(\boxx,\boy) 
\qquad \forall \ \boxx,\boy\in D \ .
\end{equation}
\end{definition}

We prove below in Lemma~\ref{L:equivalence} that, under the assumption that $\bov$ is $\widehat\boxx$-uniformly continuous, the $L$-directional Lipschitz continuity is equivalent to requiring $L$-Lipschitz continuity along forward-backward integral curves of $\bov$, in the sense that
\begin{equation}
\label{E:LrrgergGeneral1}
|\phi( \bog(s))-\phi( \bog(t))|\leq L |t-s| 
\qquad
\forall \ \bog  \text{ forward-backward integral curve of $\bov$ with $\bog(s),\,\bog(t)\in D$.}
\end{equation}

\begin{lemma}
\label{L:equivalence}
Suppose $\bov$ is $\widehat\boxx$-uniformly continuous. 
Then, a function $\phi: D \subset [0,T]\times \R^n\to \R$ is $L$-directionally Lipschitz continuous
(as in Definition~\ref{def:equivalenceGen}) if and only if it is
$L$-Lipschitz continuous along all forward-backward integral curves of $\bov$ joining points in $D$ 
(as in~\eqref{E:LrrgergGeneral1})
\end{lemma}
\begin{proof}
Suppose~\eqref{E:LrrgergGeneral1} holds.
Let $\boxx$ and $\boy$ be points in $D$.
One needs to verify the inequality~\eqref{E:LrrgergGeneral0} only if $\overline S:=d_{0}(\boxx,\boy)$ is finite, otherwise it holds trivially.
As a consequence of the characterization of $d_{0}$ in Lemma~\ref{L:equivalenceDist}, there is some forward-backward integral curve $\bog:[0,\overline S]\to\R^{N}$ of $\bov$ joining $\boxx$ and $\boy$. In particular 
\[
|\phi( \boxx)-\phi( \boy)|=|\phi( \bog(0))-\phi( \bog(\overline S))|
\stackrel{\eqref{E:LrrgergGeneral1}} {\leq}L \overline S \equiv Ld_{0}(\boxx,\boy) \ .
\]

Suppose now $\phi$ is $L$-directionally Lipschitz continuous.
If $\boxx=\bog(s)$ and $\boy=\bog(t)$ for some forward-backward integral curve $\bog$ of $\bov$, then~\eqref{E:LrrgergGeneral1} holds since
\[
|\phi( \bog(s))-\phi( \bog(t))|=|\phi( \boxx)-\phi( \boy)|
\stackrel{\eqref{E:LrrgergGeneral0}}{\leq} L d_{0}(\boxx,\boy)
\stackrel{\text{Lemma~\ref{L:equivalenceDist}}}{\leq} L|t-s| \ . \qedhere
\]
\end{proof}

\subsection{On the definition of directional Lipschitz continuity}
In Lemma~\ref{l:compar} below we argue that, if a function $\phi : [0,T]\times\R^n \to \R$ is Lipschitz continuous in the Euclidean distance, then  the $L$-Lipschitz continuity along the integral curves $\Zz(\cdot,\amf)$, for $\amf\in\Amf$, in the sense that
\begin{equation}\label{e:nicelip}
|\phi( \Zz(x_{0},\amf))-\phi( \Zz(y_{0},\amf))|\leq L |y_{0}-x_{0}| 
\qquad
\forall \, \amf\in\Amf \, ,\, \forall \, x_{0},y_{0}\in\TTT \ , 
\end{equation}
implies the $L$-Lipschitz continuity along all forward-backward integral curves of $\bov$, also along  those not selected by the multi-flow $\Zz$.
In particular, it implies the $L$-directional Lipschitz continuity by Lemma~\ref{L:equivalence}.
In Remarks~\ref{r:compar00} and~\ref{r:compar} we provide examples showing that the assumption of Lipschitz continuity for the Euclidean distance on the whole $[0,T]\times\R^n$ in general cannot be dropped, when integral curves of $\bov$ are not unique.

\begin{lemma}
\label{l:compar}
Consider a Borel function $\phi : [0,T]\times\R^n \to \R$. Then the following properties are equivalent:
\begin{enumerate}
\item\label{ite:1} $\phi$ is $M$-Lipschitz continuous for the Euclidean distance and $L$-directionally Lipschitz continuous,
\item \label{ite:0} $\phi$ is $M$-Lipschitz continuous for the Euclidean distance and~\eqref{E:LrrgergGeneral1} holds with constant $L$, 
\item\label{ite:2} $\phi$ is $M$-Lipschitz continuous for the Euclidean distance and~\eqref{e:nicelip} holds with constant $L$.
\end{enumerate}
\end{lemma}

\begin{proof}
We already proved in Lemma~\ref{L:equivalence} that~\eqref{ite:1} and~\eqref{ite:0} are equivalent, even without assuming the Lipschitz continuity for the Euclidean distance.
Clearly~\eqref{ite:0} implies~\eqref{ite:2}, since the condition in~\eqref{ite:2} only involves specific integral curves of the form $\bog^{\amf}:=\Zz(\cdot,\amf)$ for $\amf\in\Amf$. Hence we only need to show that~\eqref{ite:2} implies~\eqref{ite:0}. To this aim, we fix a forward-backward integral curve $\bog:[0,S]\to\TTT\times\R^{n}$ of $\bov$ and show that
\begin{equation}
\label{E:rggrgqggreg}
|\phi(\bog(s^{2}))- \phi(\bog(s^{1}))| \leq L \cdot (s^{2}-s^{1})
\qquad
\forall \ 0<s^{1}<s^{2}<S \ .
\end{equation}
The strategy is to approximate $\bog$ with a competitor joining $\bog(s^{1})$ and $\bog(s^{2})$ which alternates horizontal segments (junctions) and curves $\Zz(\cdot,\amf^{k,n})$ almost tangent to $\bog$ at some intermediate points.
Simplify notations by fixing $s^{1}=0$ and $s^{2}=1$, as the general case is a simple adaptation.
For $n\in\N$ we define for $k= 1,\dots,n $ the intermediate values $s^{k,n}$ in such a way that
\[
\frac{ k-1}{n}\leq s^{k,n}\leq \frac{ k }{n} \quad \text{and} \quad
\begin{cases}
\gamma_{0}\left(\frac{ k-1}{n}\right)\leq\gamma_{0}\left(s^{k,n} \right)\leq \gamma_{0}\left(\frac{k  }{n}\right)
\quad\text{and}\quad \exists\dot\gamma_{0}(s^{k,n})=1 \quad\text{or otherwise}\\
\gamma_{0}\left(\frac{ k-1}{n}\right)\geq\gamma_{0}\left(s^{k,n} \right)\geq \gamma_{0}\left(\frac{k  }{n}\right)
\quad\text{and}\quad \exists\dot\gamma_{0}(s^{k,n})=-1 \ .
\end{cases}
\]
Set moreover $s^{0,n}:=0$, $s^{n+1,n}:=1$. Set 
\[
t^{k,n} :=\gamma_{0}\left(s^{k,n} \right) 
\quad
k=0,\dots,n+1\ .
\]
Since $\Ima\Zz$ is dense we can pick correspondingly curves
\[
\bog^{k,n}(\cdot):=\Zz(\cdot,\amf^{i,n})
\qquad k=0,\dots,n+1
\]
such that $\amf^{k,n}$ are chosen in order to satisfy $\bog(0)=\bog^{0,n}(t^{0,n}) $, $\bog(1)=\bog^{n+1,n}(t^{n+1,n}) $ and
\[
\left|\bog^{k,n}( t^{k,n})-\bog(s^{k,n}) \right|\leq \sfrac{1}{n^{2}}\qquad
\text{for $k=1,\dots,n $.}
\]
For $ \sfrac{ k-1}{n}\leq s\leq   \sfrac{k+1 }{n} $ and $t$ between $\gamma_{0}\left( \sfrac{ k-1}{n}\right)$ and $\gamma_{0}\left(\sfrac{k+1 }{n} \right)$, $k=1,\dots,n$, we can estimate
\begin{align*}
|\bog( s)-\bog^{k,n}(t)|&\leq\left|\int_{s^{k,n}}^{s}\dot\gamma_{0}(r)\bov(\bog(r))\,dr-\int_{\gamma_{0}(s^{k,n})}^{t}\dot \gamma_{0}(s^{k,n})\bov(\bog^{k,n}(r)) \,dr\right|+\left|\bog^{k,n}( t^{k,n})-\bog(s^{k,n}) \right|
\\
&\leq \norm{\bov}_{\infty}\cdot\frac{2}{n}+\norm{\bov}_{\infty}\cdot\frac{2}{n}+\frac{1}{n^{2}} \leq\frac{5\norm{\bov}_{\infty}}{n} \ .
\end{align*}
Let $\rho>0$ and consider the corresponding modulus of $\widehat\boxx$-continuity $\delta=\delta(\sfrac{\rho}{2})$ of the vector field $\bov$ as in~Definition~\ref{D:cont}. If $\sfrac{5\norm{\bov}_{\infty}}{n}<\delta$ then by the area formula applied with the Lipschitz function $\gamma_{0}$
\begin{align*}
|\bog(s^{k+1,n} )-\bog^{k,n}(t^{k+1,n})|
\leq&
\left|\bog^{k,n}( t^{k,n})-\bog(s^{k,n}) \right|
\\
\notag&+
\left|\int_{\gamma_{0}([s^{k,n},s^{k+1,n}])}
\!\sum_{r\in ( \gamma_{0} )^{-1}(s)} \dot{ \gamma}_{0} (r)\bov\left( \bog\left(r \right)\right)ds -\int_{\gamma_{0}(s^{k,n})}^{\gamma_{0}(s^{k+1,n})} \!\!\!\!\!\!\!\!\!\!\!\!\bov(\bog^{k,n}(r)) \,dr\right|
\\\notag
\leq& \frac{1}{n^{2}}+ \frac{\rho}{2}\cdot |s^{k+1,n}-s^{k,n}|
\leq
\frac{1}{n^{2}} + \frac{\rho}{n}  
\ .
\end{align*}
We applied above the $\widehat\boxx$-uniform continuity of $\bov$.
In particular for $k=1,\dots,n$
\begin{align}
\label{E:sgrgrw}
|\bog^{k+1,n}(t^{k+1,n})-\bog^{k,n}(t^{k+1,n})|
\leq&
|\bog(s^{k+1,n} )-\bog^{k+1,n}(t^{k+1,n})|
+|\bog(s^{k+1,n} )-\bog^{k,n}(t^{k+1,n})|
\\\notag
\leq&
\frac{2}{n^{2}} + \frac{\rho}{n}  
\ .
\end{align}
Similarly $|\bog^{0,n}(t^{1,n})-\bog^{1,n}(t^{1,n})|\leq\sfrac{1}{n^{2}} + \sfrac{\rho}{n}$.
Hence we can now write
\begin{align*} 
 \phi(\bog(1))- \phi(\bog(0))
=& \sum_{k=0}^{n }\phi(\bog^{k,n}(t^{k+1,n}))-\phi(\bog^{k,n}(t^{k,n}))+\phi(\bog^{k+1,n}(t^{k+1,n}))- \phi(\bog^{k,n}(t^{k+1,n}))\,.
\end{align*}
We can estimate by the triangular inequality
\begin{align*}
|\phi(\bog(1))- \phi(\bog(0))| 
&\leq\sum_{k=0}^{n }\left|\phi(\bog^{k,n}(t^{k+1,n}))-\phi(\bog^{k,n}(t^{k,n}))\right| 
+\sum_{k=0}^{n }\left|\phi(\bog^{k+1,n}(t^{k+1,n}))- \phi(\bog^{k,n}(t^{k+1,n})\right| 
\end{align*}
and thus using the $M$-Lipschitz continuity of $\phi$ for the Euclidean distance and the $L$-directional Lipschitz continuity along the curves $\bog^{k,n}:=\Zz(\cdot,\amf^{k,n})$, by~\eqref{E:sgrgrw} we get
\begin{align*}
|\phi(\bog(1))- \phi(\bog(0))| 
&\leq L \sum_{k=0}^{n } ( t^{k+1,n}  - t^{k,n}  ) 
+M\sum_{k=0}^{n }\left| \bog^{k+1,n}(t^{k+1,n}) - \bog^{k,n}(t^{k+1,n} )\right| 
\\
&\leq L +  M\sum_{k=0}^{n} \left(\frac{\rho}{n}  +\frac{2}{n^{2}}\right)
  = L+ M\rho+\frac{4}{n}   \,.
\end{align*}
By the arbitrariness of both $\rho>0$ and $n\in\N$ this proves~\eqref{E:rggrgqggreg} and thus Item~\eqref{ite:1} of the statement. 
\end{proof}

\begin{remark}
\label{r:compar00}
The result in Lemma~\ref{l:compar} does not hold if the domain of $\phi$ is smaller than $[0,T]\times\R^n$, even in the case it is a flow-tube $\mcK$. Consider the continuous vector field $\bov:\R^{2}\to\R^2$ and the multi-flow
 \[
 \bov(x_{0},x_{1})=\binom{1}{3\sqrt[3]{x_{1}^{2}}} 
 \qquad
  \bog_{\amf}(x_{0})=(x_{0},(x_{0}+\amf)^{3})=:\Zz(x_{0},\amf)  
  \quad \amf\in\R\ ,
 \]
(see Figure~\ref{fig:1}). The function $\phi:\mcK\to\R$ defined by
\[
\mcK:= \bog_{0}([0,1])\cup \bog_{-1}([0,1]) \ ,
\qquad
\phi|_{\bog_{0}([0,1])}\equiv0 \ ,
\qquad
\phi|_{\bog_{-1}([0,1])}\equiv1
\]
is trivially $0$-Lipschitz continuous along $\Zz$, in the sense of~\eqref{e:nicelip}. It is however $1$-Lipschitz continuous along the integral curve of $\bov$ defined by ${\overline\bog}(x_{0},x_{1})\equiv(x_{0},0)$ and there is no extension $\widehat\phi$ of $\phi$ which is $\varepsilon$-Lipschitz continuous along such integral curve for $0<\varepsilon<1$. In particular,~\eqref{E:LrrgergGeneral1} cannot hold because $\phi$ itself does not satisfy it.
\end{remark}

\begin{figure}[h]
\centering
\begin{tikzpicture}[scale=1]
\centering 
\draw [domain=0:1,smooth,variable=\x, very thick,color=purple] plot  (\x,{ (\x-0)^3});
\draw [domain=0:1,smooth,variable=\x, very thick,color=purple] plot  (\x,{ (\x-1)^3});
\draw[->,color=black] (-.2,0) -- (1,0) node[above] {$x_0$};
\draw[->,color=black] (0,-1.3) -- (0,1.3) node[right] {$x_1$};
\end{tikzpicture}
%:
\hspace{1cm}
%:
\begin{tikzpicture}[x=.5cm, y=.5cm, style= thick]
\begin{scope}
\clip (-8, -2.5) rectangle (8, 2.5);
\foreach \xo in {-1.75,-1.25,...,9.75}
\draw[xshift=1cm,color=yellow] plot[domain=0:8] (\x,{1/4*(\x-\xo)^3});
\foreach \xo in {-9.75,-9.25,...,1.75}
\draw[xshift=1cm,color=yellow] plot[domain=0:-8] (\x,{1/4*(\x-\xo)^3});
\foreach \xo in {-1.5,-1,...,9}
\draw[xshift=1cm,color=purple] plot[domain=0:8] (\x,{1/4*(\x-\xo)^3});
\foreach \xo in {-10,-9.5,...,2}
\draw[xshift=1cm,color=purple] plot[domain=0.5:-8] (\x,{1/4*(\x-\xo)^3});
\end{scope}
\draw[->,color=black] (-6,0) -- (8.3,0) node[above] {$x_0$};
\draw[->,color=black] (1,-2.5) -- (1,2.7) node[right] {$x_1$};
\draw[color=green,style= very thick] (-6,0)--(8,0);
\end{tikzpicture}
\caption{Examples in Remarks~\ref{r:compar00} and~\ref{r:compar}. 
Left:~The domain $\mcK$ of Remark~\ref{r:compar00}.
Right:~All curves of the multi-flow $\Zz$ in Remarks~\ref{r:compar00} and~\ref{r:compar}. 
The $x_0$-axis is the image of an integral curve $\overline\bog$ of $\bov$,
and is contained in the image of the multi-flow. Nevertheless it is not parametrized by the multi-flow.}
\label{fig:1}
\end{figure}
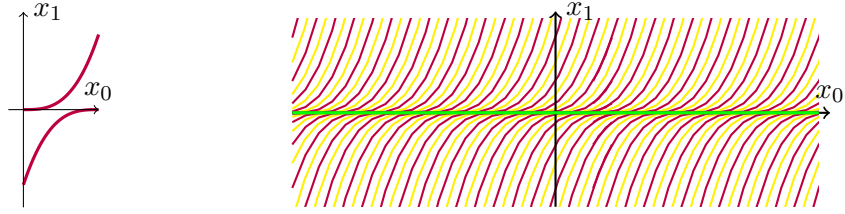

\begin{remark}
\label{r:compar}
The result in Lemma~\ref{l:compar} does not hold if we remove the assumption that $\phi$ is Lipschitz continuous for the Euclidean distance. We show by means of an example that assuming H\"older continuity of $\phi$ would not suffice. 
Consider the function
\[
\phi(x_{0},x_{1}):=\sqrt[3]{x_{0}-\sqrt[3]{x_{1}}} \,,
\]
which is constant on the integral curves $\bog^{\amf}(x_{0})=(x_{0},(x_{0}+\amf)^{3})=:\Zz(x_{0},\amf)$, for $\amf\in\R$, of the continuous vector field $\bov:\R^{2}\to\R^2$ 
 \[
 \bov(x_{0},x_{1})=\binom{1}{3\sqrt[3]{x_{1}^{2}}} \, .
 \]
However, the function $\phi$ is only H\"older continuous along the integral curve of $\bov$ 
 \[
 \overline\bog(x_{0})=\binom{x_{0}}{0} \, ,
 \] 
 which does not belong to the multi-flow $\Zz$. Indeed, there holds $\phi (\overline\bog(x_{0})) = \sqrt[3]{x_{0}}$.
\end{remark}

\section{The direction Lipschitz extension lemma}
\label{S:directionalV}

Let us introduce the notation for the Lipschitz constant of a function with respect to the distance $d_{\lambda}$.

\begin{definition}
\label{D:lipConst}
Let $\phi$ be a given function. For $0 \leq \lambda \leq 1$
we denote by $\Lipl{\lambda}{\phi}$ the best Lipschitz constant for the distance $d_{\lambda}$ of the function $\phi$,
that is
\[
\Lipl{\lambda}{\phi}:=\min\{L\ | \quad \forall \ \boxx,\boy\quad  |\phi(\boy)-\phi(\boxx)|\leq L \,d_{\lambda}(\boxx,\boy)\}\,,
\qquad
\lambda\in[0,1]\ .
\]
We denote by $\Lipl{}{\phi}$ the best Lipschitz constant for the Euclidean distance of the function $\phi$.
\end{definition}

\subsection{Convergence of the Lipschitz constants}

The next property is the key of our argument.

\begin{theorem}
\label{T:corvengenceLipConst}
Consider
\begin{itemize}
\item a vector field $\bov:\TTT\times\R^{n}\to\R^{N}$, $\widehat \boxx$-uniformly continuous as in Definition~\ref{D:cont}, with $v_{0}\equiv1$, 
\item a Borel multi-flow $\Zz:\TTT\times \Amf\to \TTT\times\R^{n}$ of $\bov$ as in Definition~\ref{d:multiflow}, with $\Ima\Zz$ dense, and 
\item a subset $A\subset\Amf$ and a flow-tube $ \mcK=\Zz(\TTT\times A)$ which is compact in $[0,T]\times\R^{n}$.
\end{itemize}
Let $\phi:\mcK\to\R$ be a function which is Lipschitz continuous for the Euclidean distance.
Then the Lipschitz constant $\Lipl{\lambda}{\phi}$ as in Definition~\ref{D:lipConst} decreases to $\Lipl{0}{\phi}$ as $\lambda\downarrow 0$.
\end{theorem}

We observe that the result in the theorem is trivial when we compare points connected by 
integral curves in the multi-flow. Indeed, if $\boxx=\Zz(x_{0},\amf)$ and $\boy=\Zz(y_{0},\amf)$ for some $\amf\in\Amf$,
then by Item~\ref{ite:smallchar} in Lemma~\ref{L:fwegeg}
\begin{equation}\label{e:L1}
\left|\phi(\boxx)-\phi(\boy) \right| \leq  \Lipl{0}{\phi} \cdot d_{0}(\boxx,\boy) =  \Lipl{0}{\phi} \cdot d_{\lambda}(\boxx,\boy)
\quad\text{for $0\leq\lambda\leq1 $,}
\end{equation}
hence $\Lipl{0}{\phi}$ is a Lipschitz constant for $\phi$ evaluated on such points.

\begin{proof}[Proof of Theorem~\ref{T:corvengenceLipConst}]
Lemma~\ref{L:fwegeg} ensures that the distances $d_{\lambda}$ are increasing for $\lambda \downarrow 0$, therefore the Lipschitz constants $\Lipl{\lambda}{\phi}$ are decreasing for $\lambda \downarrow 0$. Moreover, $\Lipl{0}{\phi} \leq \Lipl{\lambda}{\phi}$ for any $0< \lambda \leq 1$, and in particular $\Lipl{0}{\phi}\in\R$. We denote
\begin{equation}
\label{E:rgqgnjnfjf}
L:=\Lipl{0}{\phi} \ ,
\qquad \qquad
M:=\Lipl{}{\phi}\ .
\end{equation}
Fix $\varepsilon>0$. We want to prove that 
\begin{equation}\label{e:todo}
\exists \overline \lambda=\overline \lambda(\varepsilon)>0
\qquad
:\qquad \left|\phi(\boxx)-\phi(\boy) \right| <  (L+\varepsilon) \,d_{\overline\lambda_{}}(\boxx,\boy) 
\qquad \forall \ \boxx,\boy\in \mcK \,, \quad \boxx\neq\boy \ ,
\end{equation}
from which the inequality will also hold for $0\leq \lambda< \overline\lambda$ due to the fact that $d_{\lambda_{}}(\boxx,\boy)$ monotonically increases for $\lambda \downarrow 0$. We are going to prove~\eqref{e:todo} by contradiction. If~\eqref{e:todo} does not hold, then
\begin{align}
\label{E:vouofiuirb}
\forall  \lambda >0
\qquad \exists\ \boxx^{\lambda},\boy^{\lambda}\in \mcK\,,\ \boxx^{\lambda}\neq\boy^{\lambda}
\qquad
:\qquad \big|\phi(\boxx^{\lambda})-\phi(\boy^{\lambda}) \big| \geq  (L+\varepsilon) \,d_{ \lambda_{}}(\boxx^{\lambda},\boy^{\lambda}) 
\ .
\end{align}
In the two steps below we show that~\eqref{E:vouofiuirb} yields a contradiction.
In {\sc Step 1} we consider the case when there exists a sequence of values $\lambda_{k}\downarrow0$ and points $\boxx^{\lambda_{k}}$ and $\boy^{\lambda_{k}}$ satisfying~\eqref{E:vouofiuirb} and for which $|\boxx^{\lambda_{k}}-\boy^{\lambda_{k}}|\to0$.
In {\sc Step 2} we consider the case when there exists no sequence of values $\lambda_{k}\downarrow0$ and points 
$\boxx^{\lambda_{k}}$ and $\boy^{\lambda_{k}}$ as in~\eqref{E:vouofiuirb} for which $|\boxx^{\lambda_{k}}-\boy^{\lambda_{k}}|\to0$.
We remark that only in {\sc Step 1}  the continuity of the vector field plays a role.

\newcommand{\getsizes}[2]% width, height
{   \path (current bounding box.south west);
    \pgfgetlastxy{\xsw}{\ysw}
    \path (current bounding box.north east);
    \pgfgetlastxy{\xne}{\yne}
    \pgfmathsetmacro{\picwidth}{(\xne-\xsw)/28.453}
    \pgfmathsetmacro{\picheight}{(\yne-\ysw)/28.453}
    \pgfmathsetmacro{\picxscale}{#1/\picwidth}
    \pgfmathsetmacro{\picyscale}{#2/\picheight}
    \xdef\xsca{\picxscale}
    \xdef\ysca{\picyscale}
}
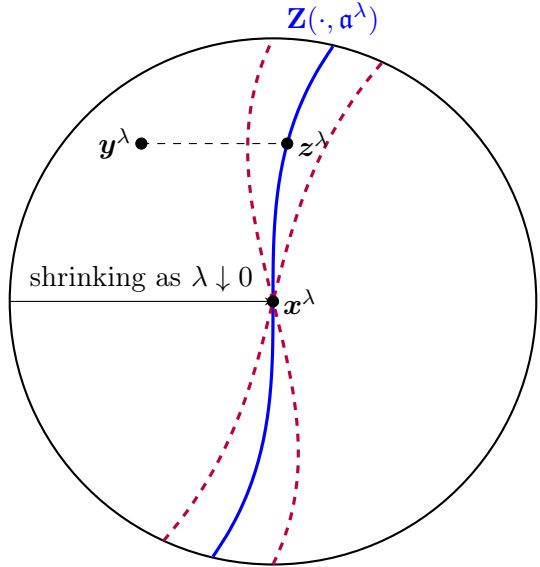
\begin{figure}[h]\centering
\begin{tikzpicture}[font=\normalsize]
\begin{scope}[scale=7, yshift=.85cm]
      \draw[thick] (0,0) circle (.5cm);
      \draw[domain=-.485:.485,smooth,variable=\y,very thick,blue] plot ({\y^3},{\y}) node[above] {$\Zz(\cdot,\amf^{\lambda})$};
      \draw[domain=-.455:.455,smooth,variable=\y,dashed,very thick,purple] plot ({\y^3+\y/4},{\y}) node[above] {};
      \draw[domain=-.5:.5,smooth,variable=\y,dashed,very thick,purple] plot ({\y^3-\y/4},{\y}) node[above] {};
      \filldraw (.027,.3)circle (.3pt) node [below, right] {$\boz^{\lambda}$};
      \filldraw (-0.25,0.3)circle (.3pt) node [below, left] {$\boy^{\lambda}$};
      \filldraw (0,0) circle (.3pt) node[below, right] {$\boxx^{\lambda}$};
      \draw[dashed] (-0.25,0.3)--(0.027,0.3);
      \draw[-stealth] (-.5,0)--(0,0) node[pos=.5,above ]{$\text{shrinking as }\lambda\downarrow0$}; 
\end{scope}
\end{tikzpicture}
%:
\caption{Step~1 in the proof of Theorem~\ref{T:corvengenceLipConst}. We define $\boz^{\lambda}$ along $\Zz(\cdot,\amf^{\lambda})$ with the same $0$-coordinate of  $\boy^{\lambda}$. 
We first prove inequality~\eqref{E:rtthhhhhrrh}, which means that $\boz^{\lambda}$ lies within the dashed region $|\boz^{\lambda}-\boy^{\lambda}| \leq  \sfrac{\varepsilon}{2M}|x_{0}^{\lambda}-y_{0}^{\lambda}| 
$, then we show that this contradicts~\eqref{E:vouofiuirb}.}
\label{fig:case3}
\end{figure}

\medskip

\noindent\underline{\sc Step 1}: Assume that there exist values $\lambda_{k}\downarrow0$ and points 
$\boxx^{\lambda_{k}}$ and $\boy^{\lambda_{k}}$ as in~\eqref{E:vouofiuirb} with 
\begin{equation}\label{e:noclose}
|\boxx^{\lambda}-\boy^{\lambda}|\to0\ .
\end{equation}
By compactness, up to subsequences $\boxx^{\lambda}$ and $\boy^{\lambda}$ are converging to some point $\overline\boxx\in\mcK$. 
Let $\rho>0$ so that
\begin{align}
\label{E:dergrigonng}
\frac{ L \rho}{L- (1+\rho) \norm{\bov}_{\infty} {M} }
<\frac{\varepsilon}{2M}
\end{align}
and denote by $\delta=\delta(\rho)$ the corresponding modulus of $\widehat\boxx$-uniformly continuity of $\bov$ as in Definition~\ref{D:cont}.

Choose points $\boz^{\lambda}\in\mcK$ with $y_{0}^{\lambda}=z_{0}^{\lambda}$ and such that $\boxx^{\lambda}=\Zz(x_{0}^{\lambda},\amf^{\lambda})$ and $\boz^{\lambda}=\Zz(y_{0}^{\lambda},\amf^{\lambda})$ (see Figure~\ref{fig:case3}). This is possible by the definition of flow-tube.
Since $\phi$ is $M$-Lipschitz continuous on $\mcK$ and $L$-directionally Lipschitz continuous, interpolating with $\pm\phi(\boz^{\lambda})$ one has the inequality
\begin{equation}
\label{E:wqfaow233bryuoawgygrw}
\left|\phi(\boxx^{\lambda})-\phi(\boy^{\lambda}) \right| \leq   L |x_{0}^{\lambda}-y_{0}^{\lambda}|+M|\boz^{\lambda}-\boy^{\lambda}| \,.
\end{equation}
If we assume that both inequalities~\eqref{E:vouofiuirb} and~\eqref{E:wqfaow233bryuoawgygrw} hold we must have
\[
d_{\lambda}(\boxx^{\lambda},\boy^{\lambda})
<
 |x_{0}^{\lambda}-y_{0}^{\lambda}|+\frac{M}{L}|\boz^{\lambda}-\boy^{\lambda}| \,.
\]
By Definition~\ref{D:distances} of $d_{\lambda}(\boxx^{\lambda},\boy^{\lambda})$ we can thus fix a competitor $\bog^{\lambda}$ joining $\boxx^{\lambda}$ and $\boy^{\lambda}$ such that
\[
\mass{\bog^{\lambda}}=\Ll^{1}\left(\Igp{\lambda}\right) + \frac{1}{\lambda}\,\Ll^{1}\left(\Ignp{\lambda}\right)
<
 |x_{0}^{\lambda}-y_{0}^{\lambda}|+\frac{M}{L}|\boz^{\lambda}-\boy^{\lambda}| \,.
\]
Multiplying by $\lambda$ and reordering the terms we get
\begin{equation}
\label{E:242q8t3huyrgobyg4}
0\leq
\lambda\left(\Ll^{1}\left(\Igp{\lambda}\right) -  |x_{0}^{\lambda}-y_{0}^{\lambda}|\right)
<
\lambda \frac{M}{L}|\boz^{\lambda}-\boy^{\lambda}| -\Ll^{1}\left(\Ignp{\lambda}\right) \,.
\end{equation}
Equation~\eqref{E:242q8t3huyrgobyg4} has the following consequences:
\begin{enumerate}
\item
\label{item:rggfgrg}
 $\Ll^{1}\left(\Ignp{\lambda}\right) < \lambda \frac{M}{L}|\boz^{\lambda}-\boy^{\lambda}|$. This means that an increasingly smaller portion of $|\boz^{\lambda}-\boy^{\lambda}|$ is run by the competitor along horizontal paths.
\item 
\label{item:rggt23553fgrg}
$\Ll^{1}\left(\Igp{\lambda}\right) \leq |x_{0}^{\lambda}-y_{0}^{\lambda}|+\frac{M}{L}|\boz^{\lambda}-\boy^{\lambda}|$. 
\item 
\label{E:modulusitem}
By the previous two points $\Ll^{1}(\Ig{\lambda})$ is vanishingly small.
Since
\begin{align*}
|\bog^{\lambda}(s)-\overline\boxx | 
&\leq 
|\bog^{\lambda}(s) -\boxx^{\lambda} |+ |\boxx^{\lambda} -\overline\boxx | 
\leq \norm{\bov}_{\infty}\Ll^{1}(\Ig{\lambda}) +  |\boxx^{\lambda} -\overline\boxx | \\
|\boz^{\lambda}  -\boxx^{\lambda} |
&\leq \norm{\bov}_{\infty}|y_{0}^{\lambda} -x_{0}^{\lambda} | \\
|\Zz(r,\amf^{\lambda})  -\boxx^{\lambda} |
&\leq \norm{\bov}_{\infty}|y_{0}^{\lambda} -x_{0}^{\lambda} |
\end{align*}
and since $\boxx^{\lambda} \to \overline \boxx$ and $\boy^{\lambda}\to \overline \boxx$, there exists $\overline \lambda$ such that for $0<\lambda<\overline\lambda $
\[
|\bog^{\lambda}(s)-\overline\boxx | + |\boxx^{\lambda}-\overline\boxx | + |\boy^{\lambda}-\overline\boxx | +|\Zz(r,\amf^{\lambda})-\overline\boxx | < \delta/2
 \qquad \text{for all $s\in\Ig{\lambda}$ and all $x_{0}^{\lambda} \leq r\leq y_{0}^{\lambda}$.}
 \]
In particular $ |\boz^{\lambda}-\overline\boxx | \leq \delta/2$.
\end{enumerate}
Since $\bog^{\lambda}$ is a competitor joining $\boxx^{\lambda}$ and $\boy^{\lambda}$, by Point~\eqref{item:rggfgrg} above 
\begin{align}
\label{E:epiporpkw}
\left | \boxx^{\lambda}- \boy^{\lambda} - \int_{\Igp{\lambda}} \dot\gamma_{0}^{\lambda}(s) \bov(\bog^{\lambda}(s))ds\right|\leq \norm{\bov}_{\infty}\Ll^{1}\left(\Ignp{\lambda}\right)
<  \lambda \norm{\bov}_{\infty} \frac{M}{L}|\boz^{\lambda}-\boy^{\lambda}|\,.
\end{align}
Since we stay within the modulus of $\widehat\boxx$-uniformly continuity of $\bov$ by~\eqref{E:modulusitem} and by the triangular inequality we have
\begin{align*}
\left | \int_{x_{0}^{\lambda}}^{y_{0}^{\lambda}}  \bov(\Zz(s,\amf^{\lambda}))ds - \int_{\Igp{\lambda}}   \dot\gamma_{0}^{\lambda}(s)\bov(\bog^{\lambda}(s))ds\right|
&=\left | \int_{\Igp{\lambda}}   \dot\gamma_{0}^{\lambda}(s)\bov(\Zz(\gamma_{0}^{\lambda}(s),\amf^{\lambda}))ds - \int_{\Igp{\lambda}}   \dot\gamma_{0}^{\lambda}(s)\bov(\bog^{\lambda}(s))ds\right|
\\
&\leq \rho \Ll^{1}\left(\Igp{k}\right) \leq \rho |x_{0}^{\lambda}-y_{0}^{\lambda}|+\rho\frac{M}{L}|\boz^{\lambda}-\boy^{\lambda}| \ ,
\end{align*}
where we also applied the estimate in Point~\eqref{item:rggt23553fgrg} above. Collecting the last two inequalities we get
\begin{align}
\label{E:epipor425245pkw}
|\boz^{\lambda}-\boy^{\lambda}|=\left | \boxx^{\lambda}- \boy^{\lambda} + \int_{x_{0}^{\lambda}}^{y_{0}^{\lambda}}  \bov(\Zz(s,\amf^{\lambda}))ds\right| 
<  \rho |x_{0}^{\lambda}-y_{0}^{\lambda}|+ (\lambda+\rho) \norm{\bov}_{\infty} \frac{M}{L}|\boz^{\lambda}-\boy^{\lambda}| 
\end{align}
and thus by the smallness of $\rho$ as in~\eqref{E:dergrigonng}
\begin{align}
\label{E:rtthhhhhrrh}
|\boz^{\lambda}-\boy^{\lambda}| \leq
\frac{ L \rho}{L- (\lambda+\rho) \norm{\bov}_{\infty} {M} } |x_{0}^{\lambda}-y_{0}^{\lambda}| 
 <  \frac{\varepsilon}{2M}|x_{0}^{\lambda}-y_{0}^{\lambda}| 
\,.
\end{align} 
Going back to the inequality~\eqref{E:wqfaow233bryuoawgygrw}, together with~\eqref{E:rtthhhhhrrh} we obtain for $0<\lambda< \overline\lambda $
\begin{align}\label{e:L4}
\left|\phi(\boxx^{\lambda})-\phi(\boy^{\lambda}) \right| &\leq   L |x_{0}^{\lambda}-y_{0}^{\lambda}|+M|\boz^{\lambda}-\boy^{\lambda}| 
\\
 &< (L+\sfrac{\varepsilon}{2})|x_{0}^{\lambda}-y_{0}^{\lambda}|  \notag
\\
& \notag\leq (L+\sfrac{\varepsilon}{2}) d_{\lambda}(\boxx^{\lambda},\boy^{\lambda}) 
\end{align}
where the last inequality is  due to Property~\eqref{ite:qqreoihr} in Lemma~\ref{L:fwegeg}.
We reached a contradiction with~\eqref{E:vouofiuirb}, therefore~\eqref{e:noclose} cannot hold. 

\medskip

\noindent\underline{\sc Step 2}: We have shown in {\sc Step 1} that for a sequence of values $\lambda_{k}\downarrow0$ and points $\boxx^{\lambda_{k}}$ and $\boy^{\lambda_{k}}$ satisfying~\eqref{E:vouofiuirb} 
 it is not possible to have $|\boxx^{\lambda_{k}}-\boy^{\lambda_{k}}|\to0$.
Since $\boxx^{\lambda_{k}}$ and $\boy^{\lambda_{k}}$ belong in the compact set~$\mcK\times\mcK$, they must thus have positive distance from the diagonal. More precisely, denoting by $\overline\delta$ the strictly positive distance among the disjoint compact sets 
\[
\clos\left(\{(\boxx^{\lambda_{k}},\boy^{\lambda_{k}})\}_{k}\right)
\qquad
\text{and}
\qquad 
\Big\{(\boxx,\boy)\in\mcK\times\mcK\ :\ |\boxx-\boy|\not=0\Big\} \ ,
\]
we have 
\begin{align}
\label{E:frpreg}
\clos\left(\{(\boxx^{\lambda_{k}},\boy^{\lambda_{k}})\}_{k}\right)\subseteq \Big\{(\boxx,\boy)\in\mcK\times\mcK\ :\ |\boxx-\boy|\geq\overline\delta\Big\}\ .
\end{align}
We are going to derive a contradiction in this case too.
We do this by proving that the real constants
\[
\ell_{\mu}:=\min\{ \ell \;\;|\;\; \left|\phi(\boxx)-\phi(\boy) \right|\leq \ell \, d_{\mu}(\boxx,\boy) \quad \text{for $\boxx,\boy\in \mcK$ s.t.~$|\boxx-\boy|\geq \overline\delta $} \}
\]
decrease to the value $\ell_{0}$, which satisfies $\ell_{0}\leq L$, defined as
\begin{equation}
\label{E:l0}
\ell_{0}:=\min\{ \ell \;\;|\;\; \left|\phi(\boxx)-\phi(\boy) \right|\leq \ell \, d_{0}(\boxx,\boy)\quad \text{for $\boxx,\boy\in \mcK$ s.t.~$|\boxx-\boy|\geq \overline\delta $}  \} .
\end{equation}
and therefore it implies that there exists $\overline\mu=\overline\mu(\varepsilon,\overline\delta)$ such that
\begin{equation}\label{e:L2}
 \left|\phi(\boxx)-\phi(\boy) \right|\leq (L+\sfrac{\varepsilon}{2}) \, d_{\mu}(\boxx,\boy) \quad \text{for $\boxx,\boy\in \mcK$ with~$|\boxx-\boy|\geq \overline\delta $ and $0\leq\mu<\overline \mu$.}
\end{equation}
Equation~\eqref{e:L2} entails that~\eqref{E:vouofiuirb} and~\eqref{E:frpreg} are incompatible, yielding the desired contradiction. 

Let us thus prove that $\ell_{\mu}\downarrow\ell_{0}$.
Since distances $d_{\mu}$ increase monotonically as $\mu\downarrow0$, the constants $\ell_{\mu}$ decrease to some value $\bar\ell$. We want to show that $\bar\ell=\ell_{0}$. 
Suppose that
\[
\text{by contradiction}
\quad
\eta_{\mu}:=\ell_{\mu}-\ell_{0}\downarrow\bar\ell-\ell_{0}=: \eta> 0 
\quad\text{as $\mu\downarrow0$.}
\]
Define for $0\leq\mu \leq 1$ the functions
\[
f_{\mu}(\boxx,\boy) :=\left|\phi(\boxx)-\phi(\boy) \right|- \ell_{0} \, d_{\mu}(\boxx,\boy)
\qquad
\text{for $\boxx,\boy\in \mcK$ with~$|\boxx-\boy|\geq \overline\delta $.}
\]
Observe that for $0<\mu\leq 1$ the functions $f_{\mu}$ are upper semicontinuous functions which decrease pointwise to the possibly $(-\infty)$-valued function $f_{0}$.
Since the domain is compact, the super-level sets $K_{t}^{\mu}:=\{f_{\mu}\geq t\}$ are compact sets decreasing as $\mu\downarrow0$. As a consequence
\[
\max_{\substack{ \boxx,\boy\in \mcK\\ |\boxx-\boy|\geq \overline\delta}} f_{\mu}(\boxx,\boy)
=
\sup_{\substack{ \boxx,\boy\in \mcK\\ |\boxx-\boy|\geq \overline\delta}} f_{\mu}(\boxx,\boy)
\quad \big\downarrow\quad 
\sup_{\substack{ \boxx,\boy\in \mcK\\ |\boxx-\boy|\geq \overline\delta}} f_{0}(\boxx,\boy)
=\max_{\substack{ \boxx,\boy\in \mcK\\ |\boxx-\boy|\geq \overline\delta}} f_{0}(\boxx,\boy)
\qquad
\text{as $\mu\downarrow0$.}
\]
We stress that for $\mu=0$ the supremum is attained by lower semicontinuity of $d_{0}$
(recall Item~\ref{ite:distanceconvergence} in Lemma~\ref{L:fwegeg}). 
Notice that by definition~\eqref{E:l0} of $\ell_{0}$ we have
\[
\max_{\substack{ \boxx,\boy\in \mcK\\ |\boxx-\boy|\geq \overline\delta}} f_{0}(\boxx,\boy)=\sup_{\substack{ \boxx,\boy\in \mcK\\ |\boxx-\boy|\geq \overline\delta}} f_{0}(\boxx,\boy) = 0 \ .
\]

We thus reach a contradiction if we show the lower bound
\begin{equation}
\label{E:gtqhwhr}
\max_{\substack{ \boxx,\boy\in \mcK\\ |\boxx-\boy|\geq \overline\delta}} f_{\mu}(\boxx,\boy)
=\sup_{\substack{ \boxx,\boy\in \mcK\\ |\boxx-\boy|\geq \overline\delta}} f_{\mu}(\boxx,\boy) \geq \frac{\eta\overline\delta}{ \norm{\bov}_{\infty} } \,.
\end{equation}
To this aim, we fix $\boxx,\boy\in \mcK$ with $|\boxx-\boy|\geq \overline\delta$ and 
\begin{equation}
\label{E:geregegeger}
f_{\mu}(\boxx,\boy)\geq0 \ ,
\quad\text{for instance s.t.}\quad
\left|\phi(\boxx)-\phi(\boy) \right|= \ell_{\mu} d_{\mu}(\boxx,\boy) 
\geq (\ell_{0} +\eta )  d_{\mu}(\boxx,\boy) \ .
\end{equation}
By~\eqref{E:rgqgrqg} and by the definition of $\ell_{\mu}$, on the domain of $f_{\mu}$ one has the estimate
\begin{align*}
f_{\mu}(\boxx,\boy)
&= \left|\phi(\boxx)-\phi(\boy) \right|- \ell_{\mu} d_{\mu}(\boxx,\boy)+( \ell_{\mu}- \ell_{0}) d_{\mu}(\boxx,\boy)
&&\text{since $\ell_{\mu}- \ell_{0}\geq\eta$}
\\
&\geq \left|\phi(\boxx)-\phi(\boy) \right|- \ell_{\mu} d_{\mu}(\boxx,\boy)+\eta d_{\mu}(\boxx,\boy)
&&\text{by~\eqref{E:rgqgrqg}}
\\
&\geq  \left|\phi(\boxx)-\phi(\boy) \right|- \ell_{\mu} d_{\mu}(\boxx,\boy)+\eta \frac{|\boy-\boxx|}{ \norm{\bov}_{\infty} } 
&&\text{as $|\boy-\boxx|\geq\overline\delta $}
\\
&\geq  \left|\phi(\boxx)-\phi(\boy) \right|- \ell_{\mu} d_{\mu}(\boxx,\boy)+ \frac{\eta\overline\delta}{ \norm{\bov}_{\infty} } 
\geq \frac{\eta\overline\delta}{ \norm{\bov}_{\infty} } 
&& \text{by~\eqref{E:geregegeger}.}
\end{align*}

We have thus shown~\eqref{E:gtqhwhr} and reached a contradiction. Hence, $\ell_{\mu} \downarrow \ell_{0}\leq L$
as $\mu\downarrow0$. 
\end{proof}

\subsection{The directional Lipschitz extension lemma}

We now state and prove our main result, the directional Lipschitz extension lemma, the proof of which relies 
on Theorem~\ref{T:corvengenceLipConst}. 

\begin{theorem}[Directional Lipschitz extension]
\label{T:LipschitzEstention}
Fix $L, M,\varepsilon>0$ and consider
\begin{itemize}
\item a vector field $\bov:\TTT\times\R^{n}\to\R^{N}$, $\widehat \boxx$-uniformly continuous as in Definition~\ref{D:cont}, with $v_{0}\equiv1$, 
\item a Borel multi-flow $\Zz:\TTT\times \Amf\to \TTT\times\R^{n}$ of $\bov$ as in Definition~\ref{d:multiflow}, with $\Ima\Zz$ dense, and 
\item a subset $A\subset\Amf$ and a flow-tube $ \mcK=\Zz(\TTT\times A)$ which is compact in $[0,T]\times\R^{n}$.
\end{itemize}
Consider a Borel function $\phi: \mcK\to\R$ which is $M$-Lipschitz continuous for the Euclidean distance and $L$-directionally Lipschitz continuous. Then $\phi$ admits an extension $\widetilde\phi_{}:[0,T]\times\R^{n}\to\R$ 
which is Lipschitz continuous for the Euclidean distance and $(L+\varepsilon)$-directionally Lipschitz continuous.
\end{theorem}

\begin{proof}
Theorem~\ref{T:corvengenceLipConst} guarantees the existence of $\overline\lambda>0$ 
small enough so that $\Lipl{\overline\lambda}{{\phi_{|_{\mcK}}}}\leq \Lipl{0}{{\phi_{|_{\mcK}}}}+\varepsilon$. 
It follows that  
\[
\Lipl{\overline\lambda}{{\phi_{|_{\mcK}}}}\leq L+\varepsilon \ .
\]
Consider a Lipschitz continuous extension 
of $\overline\phi$ to $[0,T]\times\R^{n}$, which is provided by McShane Lemma, in the distance $d_{\overline \lambda}$, that is 
$$
\overline \phi : [0,T]\times\R^{n} \to \R \,, \qquad
\left|\overline\phi(\boxx)-\overline\phi(\boy)\right| 
\leq (L+\varepsilon) d_{\overline\lambda}(\boxx,\boy)  \quad \text{for all $\boxx$, $\boy \in [0,T]\times\R^n$.}
$$
This extension $\overline\phi$ is Lipschitz continuous also for the Euclidean distance by Lemma~\ref{L:distanceisd}. Moreover, 
by Property~\eqref{ite:distanceconvergence} in Lemma~\ref{L:fwegeg} we have  
\begin{align*}
\left|\overline\phi(\boxx)-\overline\phi(\boy)\right| 
&\leq (L+\varepsilon) d_{\overline\lambda}(\boxx,\boy) 
\leq (L+\varepsilon) d_{0}(\boxx,\boy) \,. 
 \end{align*}
 Therefore the extension $\overline\phi$ is also $(L+\varepsilon)$-directionally Lipschitz continuous. 
\end{proof}

\section{Applications to the continuity equation} 

In this last section we present our applications to the continuity equation, in particular we prove Theorem~\ref{T:main}. We go back to the time-space notation as in the Introduction. We denote the maximal function in the space variable by
$$
{\mathcal M} \big( \varphi \big) (x) = \sup_{r>0} \frac{1}{ \Ll^{n}(B_r(x))} \int_{B_r(x)} |\varphi(y)|\,dy \,,
$$
and when we will deal with functions of time and space we will always consider the maximal function in the space variable only, for each fixed time. Given a vector field in the Sobolev space $W^{1,p}(\R^n)$ for almost every time, we always consider its continuous-in-space representative.

\medskip

For the classical lemma below and its corollary concerning the uniqueness of integral curves for almost every starting point we acknowledge useful discussions with Pierre-Emmanuel Jabin.
\begin{lemma}\label{l:maximal}
Let $f \in W^{1,p}_{\loc}(\R^n)$ with $p>n$. Then there exists a negligible set $N \subset \R^n$ such that
\begin{equation}\label{e:newmax}
| f( \boxx)-f( \boy)| \leq C \, | \boxx- \boy| \, \big[ {\mathcal M} \big( |Df|^p \big)( \boxx)\big]^{1/p} \qquad
\forall \ \boxx,\boy \in \R^n \setminus N \,,
\end{equation}
where the constant $C$ only depends on $n$ and $p$.
\end{lemma} 

\begin{proof} 
Let us for the time being assume $f$ to be a smooth function. We start from the well-known inequality
\begin{equation}\label{e:jab}
|f( \boxx)-f( \boy)| \leq C \int_{B( \boxx, \boy)} |Df|( \boz) \left( \frac{1}{| \boxx- \boz|^{n-1}} + \frac{1}{| \boy- \boz|^{n-1}} \right) \, d \boz \,,
\end{equation}
in which $B( \boxx, \boy)$ denotes the ball with center $(x+y)/2$ and diameter $| \boxx- \boy|$ and the constant $C$ depends 
on the space dimension only (see for instance~Lemma 3.1 in~\cite{J}). 

For the first term in the right hand side of~\eqref{e:jab} we can estimate
$$
\int_{B( \boxx, \boy)}   \frac{|Df|( \boz)}{| \boxx- \boz|^{n-1}} \, d \boz \leq C \, | \boxx- \boy| \, {\mathcal M} \big( |Df| \big)( \boxx) \leq C \, | \boxx- \boy| \, \big[ {\mathcal M} \big( |Df|^p \big) ( \boxx) \big]^{1/p} \,,
$$
where for the first inequality we refer again to~\cite{J} and the second inequality follows from the H\"older inequality.

We now deal with the second term in the right hand side of~\eqref{e:jab}. The H\"older inequality implies
\begin{equation}\label{e:stepmax}
\int_{B( \boxx, \boy)} \frac{|Df|( \boz)}{| \boy- \boz|^{n-1}} \, d \boz
\leq \left(  \int_{B( \boxx, \boy)} |Df|^p ( \boz) \, d \boz \right)^{1/p} \left(  \int_{B( \boxx, \boy)} \frac{1}{| \boy- \boz|^{p'(n-1)}} \, d \boz \right)^{1/p'} \,,
\end{equation}
where $\frac{1}{p} + \frac{1}{p'}=1$. The second factor in~\eqref{e:stepmax} can be explicitly computed as
\begin{equation}\label{e:intdim}
\int_{B( \boxx, \boy)} \frac{1}{| \boy- \boz|^{p'(n-1)}} \, d \boz = C \, | \boxx- \boy|^{(1-p')(n-1)+1} \,,
\end{equation}
where $C$ depends on $n$ and $p$. Observe that the integral in~\eqref{e:intdim}
converges if $(1-p')(n-1) > -1$, condition which corresponds to $1/p' > 1-1/n$, which is satisfied since $p>n$. Hence
going back to~\eqref{e:stepmax} 
$$
\begin{aligned}
\frac{1}{| \boxx- \boy|} \int_{B( \boxx, \boy)} \frac{|Df|( \boz)}{| \boy- \boz|^{n-1}} \, d \boz
&\leq 
C \, \frac{1}{| \boxx- \boy|^{1-n/p}} \left( \frac{1}{| \boxx- \boy|^n} \int_{B( \boxx, \boy)} |Df|^p ( \boz) \, d \boz\right)^{1/p} | \boxx- \boy|^{\frac{(1-p')(n-1)+1}{p'}} \\
&= 
C \, \left( \frac{1}{| \boxx- \boy|^n} \int_{B( \boxx, \boy)} |Df|^p ( \boz) \, d \boz\right)^{1/p}\\
&\leq 
C \big[ {\mathcal M} \big( |Df|^p \big) ( \boxx) \big]^{1/p} \,,
\end{aligned}
$$
where in the last inequality we have used the fact that $B( \boxx, \boy)$ is contained in the ball centered at $\boxx$ and with radius $| \boxx- \boy|$. We have therefore
shown~\eqref{e:newmax} for all smooth functions $f$, and by a standard approximation argument for all~$f \in W^{1,p}(\R^n)$ as claimed.
\end{proof}

\begin{corollary}
\label{c:aeuni}
Let $\bobb$ be as in~\eqref{e:bobb} and assume~$p>n$. Then for a.e.~$\widehat \boxx \in \R^n$ there is a unique integral curve of $\bobb$ starting at~$\widehat \boxx$ at time $t=0$.
\end{corollary}

\begin{proof}
Let $n < \tilde{p}<p$ be fixed. Let $\widehat \boxx \in \R^n$ and assume that there is an integral curve $\bog$ of $\bobb$ starting at $\widehat \boxx$ at time $t=0$ and satisfying
\begin{equation}\label{e:assprov}
\int_0^T \big[ {\mathcal M} \big( |D\bobb|^{\tilde p} \big) (s,\bog(s)) \big]^{1/\tilde p} \, ds < \infty.
\end{equation}
We show that $\bog$ is the unique integral curve of $\bobb$ starting at $\widehat \boxx$ at time $t=0$. We do this by implementing a sort of ``pointwise version'' of the argument in~\cite{CDL}. Let $\widetilde \bog$ be any other integral curve of $\bobb$ starting at $\widehat \boxx$ at time $t=0$. Recalling that $\tilde p>n$ we can use~Lemma~\ref{l:maximal} and estimate
$$
\frac{d}{dt} \log \left(1 + \frac{|\bog(t)-\widetilde\bog(t)|}{\delta}\right)
\leq
\frac{|\bobb(t,\bog(t)) - \bobb(t,\widetilde\bog(t))|}{|\bog(t)-\widetilde\bog(t)|}
\leq C \, \big[ {\mathcal M} \big( |D\bobb|^{\tilde p} \big) (t,\bog(t)) \big]^{1/\tilde p} \,.
$$
Integrating the above inequality in time and using that~$\bog(0)=\widetilde\bog(0)=\widehat \boxx$ we get
$$
\log \left(1 + \frac{|\bog(t)-\widetilde\bog(t)|}{\delta}\right)
\leq
C \, \int_0^T \big[ {\mathcal M} \big(  |D\bobb|^{\tilde p} \big) (s,\bog(s)) \big]^{1/\tilde p} \, ds 
\qquad \forall \ t \in \TTT \,.
$$
Letting $\delta \to 0$ and using~\eqref{e:assprov} we deduce that $\widetilde \bog = \bog$ on $\TTT$.

We recall that for a vector field as in~\eqref{e:bobb} the theory in~\cite{DPL} guarantees the existence of a unique regular Lagrangian flow~$\RLF$. We apply the above argument to the integral curves~$\bog_{\widehat \boxx} = \RLF(\cdot,\widehat \boxx)$. It only remains to check that condition~\eqref{e:assprov} is satisfied for a.e.~$\widehat \boxx \in \R^n$. For every ball $B_R(0) \subset \R^n$ we can estimate
\begin{equation}\label{e:jensen}
\begin{aligned}
\left( \int_{B_R(0)} \int_0^T \big[ {\mathcal M} \big( |D\bobb|^{\tilde p} \big) (s,\bog_{\widehat \boxx}(s)) \big]^{1/\tilde p} \, ds \, d\widehat \boxx \right)^{\tilde p}
& = 
\left( \int_{B_R(0)} \int_0^T \big[ {\mathcal M} \big( |D\bobb|^{\tilde p} \big) (s,\RLF(s,\widehat \boxx) \big]^{1/\tilde p} \, ds \, d\widehat \boxx \right)^{\tilde p} \\
& \leq 
C \, \int_0^T \int_{B_R(0)}  {\mathcal M} \big( |D\bobb|^{\tilde p} \big) (s,\RLF(s,\widehat \boxx)  \, d\widehat \boxx \, ds \\
& \leq 
C \, \int_0^T \int_{B_{R + T \|\bobb\|_\infty} (0)}  {\mathcal M} \big( |D\bobb|^{\tilde p} \big) (s,\widehat \boxx) \, d\widehat \boxx \, ds \,, \\
& \leq 
C \, \left( \int_0^T \int_{B_{R + T \|\bobb\|_\infty} (0)} \left| {\mathcal M} \big( |D\bobb|^{\tilde p} \big) (s,\widehat \boxx) \right|^{p/\tilde p} \, d\widehat \boxx \, ds \right)^{\tilde p/p}\,, \\
& \leq 
C \, \left(\int_0^T \int_{\R^n}  \left| |D\bobb|^{\tilde p}(s,\widehat \boxx) \right|^{p / \tilde p} \, d\widehat \boxx \, ds \right)^{\tilde p/p}\,,
\end{aligned}
\end{equation}
where the first inequality follows from Jensen's inequality, in the second inequality we have changed variable along~$\RLF$ and in the last inequality we have applied the strong estimate for the maximal function (using the fact that $\sfrac{p}{\tilde p}>1$). The constant $C$ in~\eqref{e:jensen} changes from line to line and depends on $R$, $T$, $p$, $\tilde p$, $n$, $\|\bobb\|_\infty$, and the compressibility constant of~$\RLF$ as in~\eqref{e:inc}.  Recalling the assumption~\eqref{e:bobb-a} we conclude that the 
right hand side in~\eqref{e:jensen} is finite, and therefore we conclude that~\eqref{e:assprov} is satisfied for a.e.~$\widehat \boxx \in \R^n$.
\end{proof}

\begin{remark} 
In~\cite[Section~5]{ACM} the authors construct a two-dimensional, bounded, divergence-free vector field satisfying~\eqref{e:bobb-a} for all $1 \leq p < \infty$, the integral curves of which are non unique for initial points on a segment (see in particular~\cite[Remark~5.1]{ACM}). This shows that the result of Corollary~\ref{c:aeuni} does not hold {\em for every} $\widehat \boxx \in \R^n$, not even for very large $p<\infty$.
\end{remark}

For a vector field~$\bobb$ be as in~\eqref{e:bobb} measure solutions to the Cauchy problem~\eqref{e:pde1} are well defined, as we consider the continuous representative of the vector field. 

\begin{corollary}\label{c:coro}
Let $\bobb$ be as in~\eqref{e:bobb} and assume~$p>n$. Fix a nonnegative function~$u_0 \in L^1_\loc(\R^n)$. 
Then the Cauchy problem~\eqref{e:pde1} has a unique solution among all positive measure solutions~$\mu_t$ with $\mu_0 = u_0 \Ll^n$, and this unique measure solution coincides with the unique solution in $L^1_\loc(\TTT\times\R^n)$. 
\end{corollary}

\begin{proof}
This is a direct consequence of Corollary~\ref{c:aeuni} and of the superposition principle for positive measure solutions of the continuity equation, see~\cite[Theorem~8.2.1]{AGS}.
\end{proof}

See also~\cite{AB,Clop,BG,BCD} for further recent results on the uniqueness of measure solutions to the continuity equation. In particular,~\cite{BCD} contains interesting extensions of Lemma~\ref{l:maximal} and Corollary~\ref{c:coro}.

\begin{remark}
It easy to see that (unless the vector field is regular) in general we cannot have uniqueness of measure valued solutions that change sign, or of positive measure solutions in case the initial datum is a positive measure with a nontrivial singular part.
\end{remark}

The almost-everywhere uniqueness proven in Corollary~\ref{c:aeuni} is in fact not sufficient in order to conclude the proof of Theorem~\ref{T:main}. Indeed, we need the almost-everywhere triviality of forward-backward integral curves. This does not follow in a straightforward way from Corollary~\ref{c:aeuni} through a ``pointwise argument'', as we explain in the following remark.

\begin{remark}\label{r:nouni}
It is possible to construct a divergence-free vector field belonging $L^\infty\left( [0,T] ; W^{1,p}(\R^n;\R^n) \right)$ with the following property: there exists $\widehat \boxx \in \R^n$ such that the
integral curve starting at $\widehat \boxx$ is unique, but there is a nontrivial forward-backward integral curve starting at $\widehat \boxx$. It is unclear to us whether this can happen for starting points in a set of positive measure.
\end{remark} 

We are now in the position of proving our main result. 

\begin{proof}[Proof of Theorem~\ref{T:main}]
We refer to the strategy introduced in~\cite{CC} and summarized in \S~\ref{s:SP}. 

For any $\varepsilon>0$ and $R>0$, the theory in~\cite{CDL} provides a compact set $K_\varepsilon \subset B_R(0)$ with $\Ll^n \big( B_R(0) \setminus K_\varepsilon \big) \leq \varepsilon$ and such that the regular Lagrangian flow is Lipschitz on~$K_\varepsilon$. Moreover, let $N \subset B_R(0)$ be the set of starting points for which the forward-backward integral curves are nontrivial. Since $N$ is negligible by assumption, we can find an open set $N_\varepsilon \supset N$ with $\Ll^n(N_\varepsilon) \leq \varepsilon$. We thus define $\widetilde{K}_\varepsilon = K_\varepsilon \setminus N_\varepsilon$ and notice that $\Ll^n \big( B_R(0) \setminus \widetilde{K}_\varepsilon \big) \leq 2\varepsilon$.

We apply the procedure described in~\cite{CC} with the set $\widetilde{K}_\varepsilon$ replacing the set~$K_\varepsilon$. We only need to check that the function~$\psi$ defined in~\cite[Lemma~3.2]{CC} is $L$-directionally Lipschitz continuous: once this has been done, the remaining steps in the proof in~\cite{CC} work without any further changes. 

In order to do this, we pick any two points $\boxx,\boy \in \RLF([0,T] \times \widetilde{K}_\varepsilon)$. If $d_0(\boxx,\boy)=+\infty$ there is nothing to check. If on the other hand $d_0(\boxx,\boy)<+\infty$, by Lemma~\ref{L:equivalenceDist} the points $\boxx$ and $\boy$ are joined by a forward-backward integral curve. However, in the set $\RLF([0,T] \times \widetilde{K}_\varepsilon)$ forward-backward integral curves are trivial, and their image is contained in the image of an integral curve in the regular Lagrangian flow. Hence we only have to check that the function~$\psi$ is Lipschitz continuous along integral curves in the regular Lagrangian flow (in the sense of~\eqref{e:nicelip}), but this is exactly what is shown in~\cite[Lemma~3.2]{CC}. 
\end{proof}

\addresseshere


\begin{thebibliography}{10}

\bibitem{ABC}
\newblock G.~Alberti, S.~Bianchini \& G.~Crippa.
\newblock {\em A uniqueness result for the continuity equation in two dimensions.}
\newblock J.~Eur.~Math.~Soc.~(JEMS)~{\bf 16} (2014), no.~2, 201--234.

\bibitem{ABCar}
\newblock G.~Alberti, S.~Bianchini \& L.~Caravenna.
\newblock {\em Eulerian, Lagrangian and Broad continuous solutions to a balance law with non-convex flux I.}
\newblock Journal of Differential Equations~{\bf 261} (2016), no.~8, 4298--4337.

\bibitem{ACM}
\newblock G.~Alberti, G.~Crippa \& A.~L.~Mazzucato.
\newblock {\em Exponential self-similar mixing by incompressible flows.}
\newblock Journal of the American Math.~Society~{\bf 32} (2019), no.~2, 445--490.

\bibitem{AMB}
\newblock L.~Ambrosio.
\newblock {\em Transport equation and Cauchy problem for BV vector fields.}
\newblock Invent. Math.~{\bf 158} (2004), no.~2, 227--260.

\bibitem{AB} 
\newblock L.~Ambrosio \& P.~Bernard.
\newblock {\em Uniqueness of signed measures solving the continuity equation for Osgood vector fields.}
\newblock Atti Accad.~Naz.~Lincei Rend.~Lincei Mat.~Appl.~{\bf 19} (2008), no.~3, 237--245.

\bibitem{AC}
\newblock L.~Ambrosio \& G.~Crippa.
\newblock {\em Continuity equations and ODE flows with non-smooth velocity.}
\newblock Proceedings of the Royal Society of Edinburgh: Section~A Mathematics {\bf 144} (2014), no.~6, 1191--1244.

\bibitem{AGS} 
\newblock L.~Ambrosio, N.~Gigli \& G.~Savar\'e.
\newblock {\em Gradient flows in metric spaces and in the space of probability measures.}
\newblock Lectures in Mathematics ETH Z\"urich, Birkh\"auser Verlag, 2008. 

\bibitem{BB}
\newblock S. Bianchini \& P. Bonicatto.
\newblock {\em A uniqueness result for the decomposition of vector fields in $\R^d$. }
\newblock Invent.~Math.~{\bf 220} (2020), 255--393.

\bibitem{CSC}
\newblock F.~Bigolin, L.~Caravenna \& F.~Serra~Cassano.
\newblock {\em Intrinsic Lipschitz graphs in Heisenberg groups and continuous solutions of a balance equation.}
\newblock Annales de l'Institut Henri Poincar\'e C, Analyse non lin\'eaire~{\bf 32} (2015), no.~5, 925--963.

\bibitem{BG} 
\newblock P.~Bonicatto \& N.~Gusev.
\newblock {\em Non-uniqueness of signed measure-valued solutions to the continuity equation in presence of a unique flow.}
\newblock Atti Accad.~Naz.~Lincei Rend.~Lincei Mat.~Appl.~{\bf 30} (2019), 511--531.

\bibitem{B}
\newblock F.~Bouchut.
\newblock Personal communication. 

\bibitem{BEA}
\newblock F.~Bouchut, R.~Eymard \& A.~Prignet.
\newblock {\em Finite volume schemes for the approximation via characteristics of linear convection equations with irregular data.}
\newblock J.~Evolution Eq.~{\bf 11} (2011), no.~3, 687--724.

\bibitem{Br}
\newblock A.~Bressan.
\newblock {\em An ill posed Cauchy problem for a hyperbolic system in 
two space dimensions.}
\newblock Rend.~Sem.~Mat.~Univ.~Padova~{\bf 110} (2003), 103--117. 

\bibitem{BCD}
\newblock E.~Bru\'e, M.~Colombo \& C.~De Lellis.
\newblock arXiv:2003.00539

\bibitem{CC}
\newblock L.~Caravenna \& G.~Crippa.
\newblock {\em Uniqueness and Lagrangianity for solutions with lack of integrability of the continuity equation.}
\newblock C.~R.~Math.~Acad.~Sci.~Paris, Ser.~I {\bf 354} (2016), no.~12, 1168--1173.

\bibitem{CL}
\newblock A.~Cheskidov \& X.~Luo.
\newblock Nonuniqueness of Weak Solutions for the Transport Equation at Critical Space Regularity.
\newblock Annals of PDE {\bf 7} (2021), Article no.~2.

\bibitem{E3}
\newblock G.~Ciampa, G.~Crippa \& S.~Spirito.
\newblock {\em Weak solutions obtained by the vortex method for the 2D Euler equations are Lagrangian and conserve the energy.}
\newblock Journal of Nonlinear Science volume {\bf 30} (2020), 2787--2820.

\bibitem{Clop}
\newblock A.~Clop, H.~Jylh\"a, J.~Mateu \& J.~Orobitg.
\newblock {\em Well-posedness for the continuity equation for vector fields with suitable modulus of continuity.}
\newblock  J.~Funct.~Anal.~{\bf 276} (2019), no.~1, 45--77.

\bibitem{CDL}
\newblock G.~Crippa \& C.~De~Lellis.
\newblock {\em Estimates and regularity results for the DiPerna-Lions flow.}
\newblock J.~Reine Angew.~Math.~{\bf 616} (2008), 15--46.

\bibitem{E2}
\newblock G.~Crippa, C.~Nobili, C.~Seis \& S.~Spirito.
\newblock {\em Eulerian and Lagrangian solutions to the continuity and Euler equations with $L^1$ vorticity.}
\newblock SIAM J.~Math.~Anal.~{\bf49} (2017), no.~5, 3973--3998.

\bibitem{E1}
\newblock G.~Crippa \& S.~Spirito.
\newblock {\em Renormalized Solutions of the 2D Euler Equations.}
\newblock Communications in Mathematical Physics {\bf 339} (2015), no.~1, 191--198.
 
\bibitem{DPL}
\newblock R.~J.~DiPerna \& P.-L.~Lions.
\newblock {\em Ordinary differential equations, transport theory and Sobolev spaces.}
\newblock Invent. Math.~{\bf 98} (1989), no.~3, 511--547.

\bibitem{J}
\newblock P.-E.~Jabin.
\newblock {\em Differential equations with singular fields.}
\newblock J. Math. Pures Appl.~(9)~{\bf  94} (2010), no.~6, 597--621. 

\bibitem{J2}
\newblock P.-E.~Jabin.
\newblock {\em Critical non-Sobolev regularity for continuity equations with
rough velocity fields.}
\newblock Journal of Differential Equations~{\bf 260} (2016), no.~5, 4739--4757. 

\bibitem{MS3}
\newblock S.~Modena \& G.~Sattig.
\newblock {\em Convex integration solutions to the transport equation with full dimensional concentration.}
\newblock Annales de l'Institut Henri Poincar\'e C, Analyse non lin\'eaire {\bf 37} (2020), no.~5, 1075--1108.

\bibitem{MS}
\newblock S.~Modena \& L.~Sz\'ekelyhidi, Jr.
\newblock {\em Non-uniqueness for the transport equation with Sobolev vector fields.}
\newblock  Ann.~PDE~{\bf 4} (2018), Article no.~18.

\bibitem{MS2}
\newblock S.~Modena \& L.~Sz\'ekelyhidi, Jr.
\newblock {\em Non-renormalized solutions to the continuity equation.}
\newblock  Calc.~Var.~Partial Differential Equations~{\bf 58} (2019), no.~6, Article no.~208.

\bibitem{SC}
\newblock F.~Serra Cassano.
\newblock {\em Some topics of geometric measure theory in Carnot groups.}
\newblock In: Geometry, analysis and dynamics on sub-Riemannian manifolds, Vol.~1, 1--121. 
EMS Ser.~Lect.~Math., Eur.~Math.~Soc., Z\"urich, 2016.

\bibitem{V}
\newblock D.~Vittone.
\newblock {\em Submanifolds in Carnot groups.} 
\newblock Tesi di Perfezionamento, Scuola Normale Superiore, Pisa, Birkh\"auser 2008.


\end{thebibliography}
\end{document}